\documentclass[12pt]{article}

\usepackage{amssymb,amsmath,amsfonts,amsthm}

\newtheorem{theorem}{Theorem}[section]
\newtheorem{proposition}{Proposition}[section]
\renewenvironment{proof}{\noindent{\bf Proof.}}{\hfill $\blacksquare$ \vspace*{4mm}}

\newcommand{\nt}{|\hspace{-0.7pt}|\hspace{-0.7pt}|}

\begin{document}



\begin{center}
{\Large \bf Local Existence of Contact Discontinuities in\\[6pt]  Relativistic Magnetohydrodynamics}
\end{center}

\begin{center}
{\large \it Yu.~L.~Trakhinin}
\end{center}

\begin{center}
Sobolev Institute of Mathematics, Novosibirsk, 630090, Russia\\
and\\
Novosibirsk State University, Novosibirsk, 630090, Russia\\[6pt]
E-mail: trakhin@math.nsc.ru
\end{center}

\vspace*{9pt}

\noindent
We study the free boundary problem for a contact discontinuity for the system of relativistic magnetohydrodynamics. A surface of contact discontinuity is a characteristic of this system with no flow across the discontinuity for which the pressure, the velocity and the magnetic field are continuous whereas the density, the entropy and the temperature may have a jump. For the two-dimensional case, we prove the local-in-time existence in Sobolev spaces
of a unique solution of the free boundary problem provided that the Rayleigh--Taylor sign condition on the jump of the normal derivative of the pressure is satisfied
at each point of the initial discontinuity.\\[9pt]

\section{Statement of the free boundary problem}
\label{s1}

We consider the equations of relativistic magnetohydrodynamics (RMHD) go\-verning the  motion of an inviscid perfectly conducting relativistic gas (in particular, plasma) in a magnetic field \cite{An,Lich}:
\begin{equation}
\nabla_{\alpha}(\rho u^{\alpha})=0, \quad
\nabla_{\alpha}T^{\alpha\beta}=0,
\quad
\nabla_{\alpha}(u^{\alpha}b^{\beta} -u^{\beta}b^{\alpha})=0,
\label{1}
\end{equation}
where $\nabla_{\alpha}$ is the covariant derivative with respect to
the Lorentzian metric $g={\rm diag}\, (-1,1,1,1)$ of the space-time with the
components $g_{\alpha\beta}$ (we restrict ourselves to the case of special relativity), $\rho$ is the proper rest-mass density of the gas, $u^{\alpha}$ are components of the 4-velocity,
\[
T^{\alpha\beta}=(\rho h +B^2)u^{\alpha}u^{\beta} +qg^{\alpha\beta} -b^{\alpha}b^{\beta},
\]
$h= 1+e +(p/\rho )$ is the relativistic specific enthalpy, $p$ is the pressure, $e=e(\rho ,S)$ is the specific internal energy, $S$ is the specific entropy, $B^2=b^{\alpha}b_{\alpha}$ are components of the magnetic field 4-vector, and $q= p+\frac{1}{2}B^2$ is the total pressure. The 4-vectors satisfy the conditions
$u^{\alpha}u_{\alpha} =-1$ and $u^{\alpha}b_{\alpha} =0$, and we choose such a scaling that the speed of the light $c=1$.

Let $(x^0,x)$ be inertial coordinates, with $t=x^0$ a time coordinate and $x=(x^1,x^2,x^3)$ space coordinates. Then
\[
u^0=-u_0=\Gamma ,\quad u^i=u_i =\Gamma v_i,\quad i=\overline{1,3},\quad u=(u_1,u_2,u_3)=\Gamma v,
\]
\[
\Gamma^2=1+|u|^2,\quad
b^0=-b_0=(u\cdot H),\quad
b^i=b_i=\frac{H_i}{\Gamma}+(u\cdot H)v_i,
\]
\[
b=(b_1,b_2,b_3),\quad
B^2=|b|^2-b_0^2=\frac{|H|^2}{\Gamma^2} +(v\cdot H)^2>0,
\]
where $\Gamma=(1-|v|^2)^{-1/2}$ is the Lorentz factor, $v=(v_1,v_2,v_3)$ is the 3--velocity and $H=(H_1,H_2,H_3)$ is the magnetic field 3-vector. Here and below we drop the superscript ${\sf T}$ while writing down a column-vector $a=(a_1,a_2,a_3)$ in the line whereas $a^{\sf T}$ appearing inside of a formula denotes a row-vector. Passing to 3-vectors, we can now rewrite the RMHD equations \eqref{1} in the form of the following system of conservation laws:
\begin{align}
& \partial_t(\rho\Gamma ) +{\rm div}\, (\rho u )=0,\label{2}\\[3pt]
& \partial_t\left(\rho h \Gamma u +|H|^2v-(v\cdot H)H\right)
+{\rm div} \left((\rho h +B^2) u\otimes u  -b\otimes b\right) + {\nabla}q=0,\label{3}\\[3pt]
& \partial_t(\rho h\Gamma^2 +|H|^2-q ) +{\rm div} \left(\rho h\Gamma u +|H|^2v-(v\cdot H)H \right)=0,\label{4}\\[3pt]
& \partial_tH -{\nabla}\times (v {\times} H)=0,\label{5}
\end{align}
where $\partial_t=\partial /\partial t$, $\nabla =(\partial_1,\partial_2,\partial_3)$ and  $\partial_i=\partial /\partial x^i$. Moreover, the first equation in the Maxwell equations containing in \eqref{1},
\begin{equation}
{\rm div}\, H=0,\label{6}
\end{equation}
should be now considered as the divergence constraint on the initial data $U (0,x)=U_0(x)$ for the unknown $U=(p,u,H,S)$.

Using \eqref{6} and the additional conservation law (entropy conservation)
\begin{equation}
\partial_t(\rho\Gamma S ) +{\rm div}\, (\rho S u )=0\label{7}
\end{equation}
which holds on smooth solutions of equations \eqref{2}--\eqref{5}, and following Godunov's symmetrization procedure \cite{Go}, we can symmetrize the conservation laws \eqref{2}--\eqref{5} in terms of a vector of canonical variables $Q=Q(U)$. This was done by  Ruggeri and Strumia \cite{RuSt} and also by Anile and Pennisi \cite{AP,An}. However, a concrete form of symmetric matrices was not found in \cite{RuSt,AP}. Moreover, if we deal with an initial-boundary value problem (especially with a free boundary problem)  it is very inconvenient and often technically impossible to work in terms of the vector $Q$.

On the one hand, from  $Q$ we can return to the vector of primitive variables $U$ keeping the symmetry property (see \cite{BT}). But, on the other hand, for RMHD finding a concrete form of symmetric matrices associated to the vector $Q$ and returning from $Q$ to $U$ is connected with unimaginable (or even almost unrealizable in practice) calculations. Another possible way of symmetrization of system \eqref{2}--\eqref{5} could be based on rewriting this system in a nonconservative symmetric form for the vector $U$ by analogy with  \cite{Tcpam} where this was done for relativistic gas dynamics. But this is also hardly realizable in practice for the cumbersome RMHD system.

To overcome this serious technical difficulty Freist\"uhler and Trakhinin \cite{FT} symmetrized the RMHD system by using the Lorentz transform. Namely, taking into account \eqref{6}, for the fluid rest frame ($v=0$) we can rewrite \eqref{2}, \eqref{3}, \eqref{5} and \eqref{7} in a nonconservative form which will be already a symmetric system. After that we should properly apply the Lorentz transform to this system to get a corresponding symmetric system in the LAB-frame. Referring to \cite{FT} for details, here we just write down the resulting symmetrization of RMHD:
\begin{equation}
\label{8}
A_0(U )\partial_tU+\sum_{j=1}^3A_j(U )\partial_jU=0,
\end{equation}
where
\[
A_0=\left(
\begin{array}{cccc}
{\displaystyle\frac{\Gamma}{\rho a^2}} & v^{\sf T} &0 &0 \\[3pt]
v & \mathcal{A} &0 &0 \\
0& 0 & \mathcal{M} &0 \\
0 &0&0 & 1
\end{array}\right),\qquad A_j=\left(
\begin{array}{cccc}
{\displaystyle\frac{u_j}{\rho a^2}} & e_j^{\sf T} &0 &0\\[3pt]
e_j & \mathcal{A}_j &{\mathcal{N}_j}^{\sf T} &0\\
0& \mathcal{N}_j & v_j\mathcal{M} &0 \\
0&0&0&v_j
\end{array}\right),
\]
\[
a^2=p_{\rho}(\rho ,S),\quad e_j= (\delta_{1j},\delta_{2j},\delta_{3j}),
\]
\[
\begin{split}
\mathcal{A}=\bigg( \rho h & \Gamma +\frac{|H|^2}{\Gamma} \bigg) I - \bigg( \rho h \Gamma + \frac{|H|^2+B^2}{\Gamma} \bigg)  v\otimes v   \\  & -\frac{1}{\Gamma}\, H\otimes H +
\frac{(v\cdot H)}{\Gamma}
\bigl( v\otimes H + H\otimes v\bigr) ,
\end{split}
\]
\[
\mathcal{M} =\frac{1}{\Gamma}\, (I +u\otimes u),\quad \mathcal{N}_j= \frac{1}{\Gamma}\,b\otimes e_j -\frac{v_j}{\Gamma}\, b\otimes v - \frac{H_j}{\Gamma^2}\,I,
\]
\[
\begin{split}
\mathcal{A}_j=
v_j& \left\{ \left( \rho h \Gamma +\frac{|H|^2}{\Gamma} \right) I - \left( \rho h \Gamma +\frac{|H|^2-B^2}{\Gamma}\right) v\otimes v  -\frac{1}{\Gamma}\, H\otimes H\right\}\\
& +\frac{H_j}{\Gamma}\left\{ \frac{1}{\Gamma^2}
\bigl( v\otimes H + H\otimes v\bigr) -2(v\cdot H)(I-v\otimes v) \right\}
\\
& +\frac{(v\cdot H)}{\Gamma}\left( H\otimes e_j + e_j \otimes H\right) -\frac{B^2}{\Gamma}
\left( v\otimes e_j + e_j \otimes v\right),
\end{split}
\]
and $I$ is the unit matrix. Clearly, $A_{\alpha}$ ($\alpha = \overline{0,3}$) are symmetric matrices. Note also that the last equation in \eqref{8} is just the nonconservative form ${\rm d}S/{\rm d}t=0$ of \eqref{7}, with ${\rm d}/{\rm d}t=\partial_t +(v,\nabla )$.

As is known \cite{An}, natural physical restrictions guaranteeing the hyperbolicity of the RMHD system do not depend on the magnetic field and coincide with corresponding ones in relativistic gas dynamics \cite{Tcpam}. In our case, by direct calculations one can show that the hyperbolicity condition $A_0>0$ holds provided that
\begin{equation}
\rho >0,\quad p_{\rho} >0,\quad 0<c_s^2<1\label{9}
\end{equation}
(of course, by default we also assume that $|v|<1$), where $c_s$ is the relativistic speed of sound, $c_s^2= a^2/h=p_{\rho}/h$. The last inequality in \eqref{9} is the relativistic causality condition.

We will consider the RMHD equations \eqref{1}/\eqref{8} for $t\in [0,T]$ in the whole space $\mathbb{R}^3$. Let
\[
\Gamma (t)=\{ x_1-\varphi (t,x')=0\}
\]
be a smooth hypersurface in $[0,T]\times\mathbb{R}^3$, with $x'=(x_2,x_3)$. Let $\Gamma (t)$ be a surface of strong discontinuity for the conservation laws (\ref{1}).
That is, we are interested in solutions of (\ref{1}) which are smooth in
\[
  \Omega^{\pm}(t)=\{\pm (x_1- \varphi (t,x'))>0\}
\]
whereas $U$ may have a jump on the discontinuity $\Gamma$. As is known, to be weak solutions of the conservation laws \eqref{1} in the whole space $\mathbb{R}^3=\Omega^+(t)\cup\Omega^-(t)$ such piecewise smooth solutions should satisfy standard jump conditions (like Rankine--Hugoniot conditions) at each point of $\Gamma (t)$:
\begin{equation}
 [\mathfrak{j}]=0,\label{10}
\end{equation}
\begin{equation}
 \mathfrak{j}\left[\left(h\Gamma +\frac{|H|^2}{\rho\Gamma}\right) v_n +\frac{(v\cdot H)}{\rho\Gamma}H_n\right]-
H_n\left[ \frac{H_n}{\Gamma^2}+(v\cdot H)v_n\right] +[q]=0,\label{11}
\end{equation}
\begin{equation}
 \mathfrak{j}\left[\left(h\Gamma +\frac{|H|^2}{\rho\Gamma}\right) v_{\tau} +\frac{(v\cdot H)}{\rho\Gamma}H_{\tau}\right]-
H_n\left[ \frac{H_{\tau}}{\Gamma^2}+(v\cdot H)v_{\tau}\right]=0,\label{12}
\end{equation}
\begin{equation}
 [H_n]=0,\label{13}
\end{equation}
\begin{equation}
 \mathfrak{j}\left[ \frac{H_{\tau}}{\rho\Gamma}\right]-H_n[v_{\tau}]=0,\label{14}
\end{equation}
\begin{equation}
\mathfrak{j}\left[ h\Gamma +|H|^2 +\frac{q}{\rho\Gamma}\right]+H_n[v\cdot H]+ [v_nq]=0,\label{15}
\end{equation}
where $[g]=g^+_{|\Gamma}-g^-_{|\Gamma}$ denotes the jump of $g$, with $g^{\pm}:=g$ in $\Omega^{\pm}(t)$,
\[
\mathfrak{j}^{\pm}=\rho^{\pm}\Gamma^{\pm} (v_n^{\pm}-\sigma),\quad v_n^{\pm}={v}^{\pm} \cdot n,\quad \mathfrak{j}:=\mathfrak{j}^{\pm}|_{\Gamma},
\]
\[
n=\frac{1}{\sqrt{1+(\partial_2\varphi )^2+(\partial_3\varphi )^2}}\,(1,-\partial_2\varphi,-\partial_3\varphi ),\quad
\sigma =\frac{\partial_t\varphi}{\sqrt{1+(\partial_2\varphi )^2+(\partial_3\varphi )^2}},
\]
\[
H_n^{\pm}=H^{\pm} \cdot n,\quad H_n:=H_n^{\pm}|_{\Gamma},
\]
\[
{v}^{\pm}_{\tau}=(v^{\pm}_{\tau _1},v^{\pm}_{\tau _2}),\quad
v^{\pm}_{\tau _i}=({v}^{\pm} {\cdot}{\tau}_i),\quad {\tau}_1=(\partial_2\varphi,1,0),\quad
{\tau}_2=(\partial_3\varphi,0,1),
\]
\[
{H}^{\pm}_{\tau}= (H^{\pm}_{\tau _1}, H^{\pm}_{\tau _2}),\quad H^{\pm}_{\tau_i}=({H}^{\pm} {\cdot}{\tau}_i).
\]

As in nonrelativistic MHD \cite{BT,LL,MTT1,MTT2}, a discontinuity with zero mass transfer flux and the magnetic field being nonparallel to it at its points ($\mathfrak{j}=0$ and $H_n\neq 0$) is called {\it contact discontinuity}. Then, for such a discontinuity one has $\sigma =v_n^{\pm}|_{\Gamma}$, i.e., $[v_n]=0$, and it follows from \eqref{14} that $[v_{\tau}]=0$. Since $[v]=0$, from \eqref{12} we get
\begin{equation}
(1-\sigma^2)[H_\tau]=0.
\label{16}
\end{equation}
Since the discontinuity speed $\sigma$ should be less than the speed of light, we have $1-\sigma^2>0$ and \eqref{16} implies $[H_\tau]=0$. In view of \eqref{13}, we thus obtain $[H]=0$. Then, it follows from \eqref{11} that $[p]$=0 and \eqref{15} hold automatically. That is, as in the nonrelativistic case \cite{BT,MTT2}, we obtain the following boundary conditions on a contact discontinuity:
\begin{equation}
[p]=0,\quad [v]=0,\quad [H]=0,\quad \partial_t\varphi=v_N^+|_{\Gamma},
\label{17}
\end{equation}
where $v_N^+ =v_1^+-v_2^+\partial_2\varphi  -v_3^+\partial_3\varphi $. We thus conclude that the contact discontinuity moves with the velocity of fluid particles and
the pressure, the velocity and the magnetic field are continuous whereas the density and the entropy (and also the temperature  $T=e_S(\rho ,S)$) may have a jump.
Here and below it is more convenient to write down the boundary conditions in terms of the vector $v=u/\Gamma =u/\sqrt{1+|u|^2}$ (it is clear that the conditions $[v]=0$  and $[u]=0$ are equivalent).

 As is noted in \cite{Goed}, the boundary conditions like \eqref{17} are most typical in the solar wind and astrophysical plasmas (i.e., in plasmas outside the solar system). Contact discontinuities are usually observed behind astrophysical shock waves bounding supernova remnants or due to the interaction of multiple shock waves driven by fast coronal mass ejections. If characteristic speeds are compatible with the speed of light, then one needs to consider contact discontinuities within the frameworks of relativistic models, in particular, RMHD.

From the mathematical point of view the free boundary problem for a contact discontinuity is the problem of seeking solutions of the systems
\begin{equation}
\label{18}
A_0(U^{\pm} )\partial_tU^{\pm}+\sum_{j=1}^3A_j(U^{\pm} )\partial_jU^{\pm}=0\quad \mbox{for}\ x\in\Omega^{\pm}(t)
\end{equation}
(cf. \eqref{8}) satisfying the boundary conditions \eqref{17} on $\Gamma (t)$ and the initial data
\begin{equation}
U^{\pm} (0,{x})=U_0^{\pm}({x}),\quad {x}\in \Omega^{\pm} (0),\quad \varphi (0,{x}')=\varphi _0({x}'),\quad {x}'\in\mathbb{R}^2
\label{19}
\end{equation}
for $t=0$.

Our main goal is to find conditions on the initial data \eqref{19} providing the existence and uniqueness in Sobolev spaces on some time interval $[0,T]$ of a solution $(U^+,U^-,\varphi )$ to the free boundary problem \eqref{17}--\eqref{19}. These conditions will be additional ones to the hyperbolicity condition \eqref{9} and the requirement $H_n\neq 0$.

Running ahead, we note that to prove the theorem on the existence and uniqueness of a smooth solution to problem \eqref{17}--\eqref{19} we should slightly revise certain places in the proof of the analogous theorem for contact discontinuities in nonrelativistic MHD  \cite{MTT1,MTT2}. In \cite{MTT2} this theorem was managed to be proved only for the 2D version of the problem, more precisely, for a 2D planar flow. Following \cite{MTT1,MTT2}, for RMHD we will also consider such a flow. This means that the flow is $x_3$-invariant, i.e., $U=U(t,x_1,x_2)$, but the velocity and the magnetic field are shearless ($v_3=H_3=0$). For nonrelativistic MHD the requirement   $v_3|_{t=0}=H_3|_{t=0}=0$ on the initial data guarantees that the flow is planar. It is easy to prove that the same is true for RMHD.

Indeed, as is known \cite{Kato,VH}, any symmetric hyperbolic system has a unique solution on a short time interval $[0,T]$. Then, it is natural to assume that system \eqref{1}/\eqref{8} has smooth solutions for $t\in [0,T]$. Omitting technical calculations, we rewrite the fourth and last equations of \eqref{1} for $U=U(t,x_1,x_2)$ as the symmetric hyperbolic system
\[
\begin{pmatrix} \Gamma^2 (\rho h\Gamma^2 +|H|^2) & 0\\ 0& 1\end{pmatrix} \partial_tW+\sum_{j=1}^2
\begin{pmatrix} a_1v_j+a_2H_j& -H_j \\-H_j& v_j\end{pmatrix}\partial_jW +BW=0
\]
for $W=(v_3,H_3)$, where a concrete form of the functions $a_1=a_1(U)$ and $a_2=a_2(U)$ and the matrix $B=B(U,U_t,\partial_1U,\partial_2U)$ is of no interest. The matrix by $\partial_tW$ is positive definite under the hyperbolicity condition \eqref{9}. Hence, the a priori $L^2$ estimate for the solutions $W$ of this system and the assumption $W|_{t=0}=0$ imply $W= 0$ for $t\in [0,T]$. It is thus physically relevant to consider 2D planar flows in RMHD.

Considering problem \eqref{17}--\eqref{19} for a 2D planar flow, without loss of generality, we will below assume that the space variable, the velocity  and the magnetic field have two components: $x=(x_1,x_2)\in \mathbb{R}^2$, $v^{\pm}=(v_1^{\pm},v_2^{\pm})\in \mathbb{R}^2$ and $H^{\pm}=(H_1^{\pm},H_2^{\pm})\in \mathbb{R}^2$. Then
\[
v_{N}^{\pm}={v}_1^{\pm}-v_2^{\pm}\partial_2\varphi,\quad H_{N}^{\pm}=H_1^{\pm}-H_2^{\pm}\partial_2\varphi,\quad
v^{\pm}_{\tau}=v_1^{\pm}\partial_2\varphi +v_2^{\pm},
\]
etc.

We note that for contact discontinuity in nonrelativistic MHD the results in \cite{MTT1} (for the linearized problem) and \cite{MTT2} (for the original nonlinear problem) were obtained for such equations of state for which $[p]=0$ implies $[\rho p_{\rho}]=0$. This is, in particular, true for the equation of state of a {\it polytropic gas}
\begin{equation}
\rho(p,S)= A p^{\frac{1}{\gamma}} e^{-\frac{S}{\gamma}}, \qquad A>0,\quad \gamma>1.
\label{pg}
\end{equation}
Let the relativistic gas be polytropic (exactly the case of a polytropic case was actually considered in \cite{MTT1,MTT2}). We note that in astrophysics this is a usual assumption for a relativistic (or nonrelativistic) gas. For a polytropic gas the first two hyperbolicity conditions in \eqref{9} are reduced to the requirement
\begin{equation}
p>0.
\label{9'}
\end{equation}
The fulfilment of \eqref{9'} guarantees the validity of the inequality $c_s^2>0$. At the same time, the causality condition $c_s^2<1$ is reduced to
\[
1+ \frac{\gamma (2-\gamma )p}{(\gamma -1)\rho} >0.
\]
Under condition \eqref{9'} the last inequality automatically holds for $\gamma \leq 2$ (it is assumed by default that $\gamma >1$, cf. \eqref{pg}). There is no need to assume that the adiabatic index $\gamma \leq 2$, but we will do this for simplicity of references to \cite{MTT1,MTT2}. That is, it will be below assumed by default that $\gamma \leq 2$. Then, as in \cite{MTT1,MTT2}, the hyperbolicity condition is reduced to the sole inequality \eqref{9'}. In astrophysics, it is usually assumed that the gas is monoatamic. This means that  $\gamma =5/3$ for a nonrelativistic gas and $\gamma =4/3$ for a relativistic one. In other words, the condition
 $\gamma \leq 2$ is obviously fulfilled.

\section{Reduced nonlinear problem in a fixed domain}

To reduce our free boundary problem (for the 2D case) to that in a fixed domain we straighten the interface $\Gamma$ by using the same simplest change of independent variables as in \cite{T09,Tcpam}. That is, the unknowns $U^+$ and $U^-$ being smooth in $\Omega^{\pm}(t)$ are replaced by the vector-functions
\begin{equation}
\widetilde{U}^{\pm}(t,x ):= {U}^{\pm}(t,\Phi^{\pm} (t,x),x')
\label{change}
\end{equation}
which are smooth in the half-plain $\mathbb{R}^2_+=\{x_1>0,\ x_2\in \mathbb{R}\}$,
where
\begin{equation}
\Phi^{\pm}(t,x ):= \pm x_1+\Psi^{\pm}(t,x ),\quad \Psi^{\pm}(t,x ):= \chi (\pm x_1)\varphi (t,x_2),
\label{change2}
\end{equation}
and the cut-off $\chi\in C^{\infty}_0(\mathbb{R})$ equals to 1 on $[-1,1]$, and $\|\chi'\|_{L_{\infty}(\mathbb{R})}<1/2$. We use the cut-off function $\chi$ to avoid assumptions about compact support of the initial data in our existence theorem. The change of variable \eqref{change} is admissible if $\partial_1\Phi^{\pm}\neq 0$. The latter is guaranteed, namely, the inequalities $\partial_1\Phi^+> 0$ and $\partial_1\Phi^-< 0$ are fulfilled, if we consider solutions for which
\begin{equation}
\|\varphi\|_{L_{\infty}([0,T]\times\mathbb{R})}\leq 1.
\label{fi}
\end{equation}
This holds if, without loss of generality, we consider the initial data satisfying $\|\varphi_0\|_{L_{\infty}(\mathbb{R})}\leq 1/2$, and the time $T$ in our existence theorem is sufficiently small.

Dropping for convenience tildes in $\widetilde{U}^{\pm}$, we reduce \eqref{17}--\eqref{19} (for the 2D case) to the initial-boundary value problem
\begin{equation}
A_0(U^\pm )\partial_tU^\pm +\widetilde{A}_1(U^\pm ,\Psi^\pm )\partial_1U^\pm +A_2(U^\pm )\partial_2U^\pm =0\quad\mbox{in}\ [0,T]\times \mathbb{R}^2_+,\label{11r}
\end{equation}
\begin{equation}
[p]=0,\quad [v]=0,\quad [H]=0,\quad \partial_t\varphi =v^+_N\quad\mbox{on}\ [0,T]\times\{x_1=0\}\times\mathbb{R},\label{12r}
\end{equation}
\begin{equation}
U^+_{|t=0}=U^+_0,\quad U^-_{|t=0}=U^-_0\quad\mbox{in}\ \mathbb{R}^2_+,
\qquad \varphi |_{t=0}=\varphi_0\quad \mbox{in}\ \mathbb{R},\label{13r}
\end{equation}
where
\[
\widetilde{A}_1(U^{\pm},\Psi^{\pm})=\frac{1}{\partial_1\Phi^{\pm}}\Bigl(
A_1(U^{\pm})-A_0(U^{\pm})\partial_t\Psi^{\pm}-A_2(U^{\pm})\partial_2\Psi^{\pm}\Bigr)
\]
($\partial_1\Phi^{\pm}=\pm 1 +\partial_1\Psi^{\pm}$) and as in (\ref{12r}) we use the notation $[g]:=g^+_{|x_1=0}-g^-_{|x_1=0}$ for any  pair of values $g^+$ and $g^-$.

We are interested in smooth solutions $(U^+,U^-,\varphi )$ of problem \eqref{11r}--\eqref{13r}, to be exact, as in \cite{MTT1,MTT2}, we are going to prove their existence under the fulfillment of the hyperbolicity condition \eqref{9'}, the requirements $|v|<1$ and  $H_n\neq 0$ and the {\it Rayleigh--Taylor sign condition} $[\partial p/\partial n]<0$ at the first moment. In other words, the initial data should satisfy the inequalities
\begin{equation}
p^\pm \geq \bar{p} >0,
\label{5.1}
\end{equation}
\begin{equation}
1-|v^{\pm}| \geq \nu >0,
\label{5.1"}
\end{equation}
\begin{equation}
|H_N^\pm|_{x_1=0}|\geq \kappa >0
\label{mf.1}
\end{equation}
\begin{equation}
[\partial_1p]\geq \epsilon >0,
\label{RT1}
\end{equation}
where $\bar{p}$, $\nu$, $\kappa$ and $\epsilon$ are positive constants, \eqref{RT1} is the Rayleigh--Taylor sign condition written for the ``straightened'' discontinuity (with the equation $x_1=0$). Since the moving domains $\Omega^\pm (t)$ were reduced to the same half-plain $\mathbb{R}^2_+$ (but not to the different
half-plains $\mathbb{R}^2_+$ and $\mathbb{R}^2_-$), the jump of the normal derivative is defined as follows:
\begin{equation}
[\partial_1a]:=\partial_1a^+_{|x_1=0}+\partial_1a^-_{|x_1=0}.
\label{norm_jump}
\end{equation}

Since the domain $\mathbb{R}^2_+$ is unbounded, then the functions belonging to Sobolev spaces $H^s(\mathbb{R}^2_+)$ must vanish at infinity. Therefore, if we want to have smooth solutions belonging to Sobolev spaces, then not $U^+$ and $U^-$ themselves but corresponding functions shifted to some smooth bounded functions $\bar{U}^+$ and $\bar{U}^-$ should belong to $H^s$. Indeed, conditions \eqref{5.1}, \eqref{mf.1} and \eqref{RT1} cannot be satisfied for $U^+$ and $U^-$ vanishing at infinity. As the shifts $\bar{U}^\pm$  one can consider magnetohydrostatic (MHS) equilibria $\bar{U}^\pm (x)$ which are smooth bounded solutions of the stationary RMHD system satisfying condition \eqref{5.1} in the whole half-plain  $\mathbb{R}^2_+$ and conditions \eqref{mf.1} and \eqref{RT1} on its boundary $x_1=0$. For nonrelativistic MHD, if we take gravity into account (that is natural when one deals with the Rayleigh--Taylor sign condition) and introduce in the MHD equations corresponding gravitational lower-order terms, then  one can present rather simple MHS equilibria satisfying conditions \eqref{5.1}, \eqref{mf.1} and \eqref{RT1} when the contact discontinuity is located between two perfectly conducting rigid walls (for the 2D case these walls are the lines $x_1=\pm a$, where $a={\rm const}$, see \cite{MTT2}). Analogous MHS equilibria can be found for RMHD as well.

Since we study the well-posedness of the problem but not the stability of its solutions, the gravitational lower-order (non-differential) terms are of no interest. This is why, as in \cite{MTT2}, for technical simplicity we do not introduce the gravitational terms in the equations and shift the solutions to the mentioned MHS equilibria. Instead of this we consider periodical boundary conditions in the tangential direction. Let
\[
D=\{x\in\mathbb{R}^2\,|\, x_1\in\mathbb{R},\ x_2\in \mathbb{T}\}
\]
be the flow domain of a relativistic gas, where $\mathbb{T}$ is the 1-torus (the unit circle), i.e., the values of a solution should coincide in the ends of the unit segment on the $x_2$-axis. Then
\[
\Gamma (t)=\{x\in \mathbb{R}\times \mathbb{T},\ x_1=\varphi (t,x_2)\},\qquad t\in [0,T].
\]
As above, we now do the change of variables \eqref{change} reducing our free boundary problem to that in the fixed domain
\[
\Omega=\{x_1>0,\ x_2\in\mathbb{T}\}
\]
with the straightened contact discontinuity
\[
\partial\Omega=\{x_1=0,\ x_2\in\mathbb{T}\}.
\]

Since $\Omega$ is still an unbounded domain, for satisfying the hyperbolicity condition \eqref{5.1} as $x_1\rightarrow +\infty$ we make the following usual shift of the unknowns to constants:
\begin{equation}
\breve{U}^\pm =U^\pm -\bar{U}^\pm ,
\label{shift}
\end{equation}
where $\bar{U}^\pm =(\bar{p},0,0,\bar{S}^\pm )$, the constants $\bar{S}^\pm$ are such that $\bar{S}^+\neq\bar{S}^-$, and $\bar{p}$ is the constant from \eqref{5.1}. After the change of unknowns \eqref{shift} the boundary conditions \eqref{12r} stay unchanged whereas
in the RMHD systems \eqref{11r} we should shift the arguments $U^\pm$ of the matrix functions to the constant vectors $\bar{U}^\pm$. Dropping for convenience the breve accents in $\breve{U}^\pm$, we get the following initial-boundary value problem in the space-time domain $[0,T]\times\Omega$:
\begin{equation}
\mathbb{L}(U^+,\Psi^+)=0,\quad \mathbb{L}(U^-,\Psi^-)=0\quad\mbox{in}\ [0,T]\times \Omega,\label{11.1}
\end{equation}
\begin{equation}
\mathbb{B}(U^+,U^-,\varphi )=0\quad\mbox{on}\ [0,T]\times\partial\Omega,\label{12.1}
\end{equation}
\begin{equation}
U^+_{|t=0}=U^+_0,\quad U^-_{|t=0}=U^-_0\quad\mbox{in}\ \Omega,
\qquad \varphi |_{t=0}=\varphi_0\quad \mbox{on}\ \partial\Omega,\label{13.1}
\end{equation}
where $\mathbb{L}(U^\pm,\Psi^\pm)=L(U^\pm,\Psi^\pm)U^\pm$,
\[
L(U^\pm,\Psi^\pm)=A_0(U^\pm +\bar{U}^\pm)\partial_t +\widetilde{A}_1(U^\pm +\bar{U}^\pm,\Psi^\pm)\partial_1+A_2(U^\pm  +\bar{U}^\pm)\partial_2,
\]
and \eqref{12.1} is the compact form of the boundary conditions
\begin{equation}
[p]=0,\quad [v]=0,\quad [H_{\tau}]=0,\quad \partial_t\varphi-v_{N}^+|_{x_1=0}=0.
\label{12'}
\end{equation}
Here and below
\[
{H}^{\pm}_{\tau}= H_1^\pm \partial_2\Psi^{\pm}+H_2^{\pm},\quad H_N^{\pm} = H_1^\pm -H_2^{\pm}\partial_2\Psi^{\pm},\quad
v_N^{\pm} = v_1^\pm -v_2^{\pm}\partial_2\Psi^{\pm}
\]
($v_N^{\pm}|_{x_1=0}= (v_1^\pm -v_2^{\pm}\partial_2\varphi )|_{x_1=0}$, etc.). Moreover, in view of \eqref{shift}, the hyperbolicity conditions for systems \eqref{11.1} can be written, for example, as (cf. \eqref{5.1})
\begin{equation}
p^\pm > -\bar{p}/4
\label{5.1'}
\end{equation}
whereas conditions \eqref{mf.1} and \eqref{RT1} for the new (``shifted'') unknowns stay unchanged.

Note that the continuity of the magnetic field $[H]=0$ is equivalent to $[H_N]=0$ and $[H_\tau ]=0$. Since the Maxwell equations \eqref{5} totally coincide with their nonrelativistic version, following \cite{MTT1}, we make the conclusion that the condition $[H_N]=0$ is not a real boundary condition and must be regarded as a restriction on the initial data \eqref{13.1} (and this is why we did not include it into the boundary conditions \eqref{12'}). More precisely, we have the following proposition.

\begin{proposition}
Let the initial data \eqref{13.1} satisfy
\begin{equation}
{\rm div}\, \mathfrak{h}^+=0,\quad {\rm div}\, \mathfrak{h}^-=0
\label{14.1}	
\end{equation}
and the condition
\begin{equation}
[H_{N}]=0,
\label{15.1}
\end{equation}
where $\mathfrak{h}^{\pm}=(H_{N}^{\pm},H_2^{\pm}\partial_1\Phi^{\pm})$. If problem \eqref{11.1}--\eqref{13.1} has a sufficiently smooth solution, then this solution satisfies \eqref{14.1} and \eqref{15.1} for all $t\in [0,T]$.
\label{p1}
\end{proposition}

As is well known (see, e.g., \cite{BT}), the condition necessary for the well-posedness of a hyperbolic problem like  \eqref{11.1}--\eqref{13.1} is the requirement that the number of the boundary conditions at a point $(t^*,x_2^*)\in [0,T]\times\partial\Omega$ of the boundary should be one unit greater than the number of outgoing (from the boundary $x_1=0$) characteristics of the 1D system
\[
\mathfrak{A}_0 \partial_t\begin{pmatrix} U^+ \\ U^- \end{pmatrix}+\mathfrak{A}_1 \partial_1\begin{pmatrix} U^+ \\ U^- \end{pmatrix}\quad \mbox{for}\ x_1>0,
\]
for which the coefficients of the matrices
\begin{equation}
\mathfrak{A}_0 =\begin{pmatrix} A_0^+ & 0 \\ 0 & A_0^- \end{pmatrix}\quad \mbox{and}\quad
\mathfrak{A}_1= \begin{pmatrix} \widetilde{A}_1^+ & 0 \\ 0 & -\widetilde{A}_1^- \end{pmatrix}
\label{A1}
\end{equation}
are fixed (``frozen'') at the mentioned point $(t^*,x_2^*)$, where
\[
\begin{split}
A_0^{\pm} := & A_0(U^{\pm}+\bar{U}^\pm)|_{x_1=0},\\
\widetilde{A}_1^{\pm} := & \pm\widetilde{A}_1(U^{\pm}+\bar{U}^\pm,\Psi^{\pm})|_{x_1=0} \\ = &\left.\left(A_1(U^{\pm}+\bar{U}^\pm) -A_0(U^{\pm}+\bar{U}^\pm)\partial_t\varphi - A_2(U^{\pm}+\bar{U}^\pm)\partial_2\varphi \right)\right|_{x_1=0}.
\end{split}
\]
This condition is dictated by the general rule of statement of boundary value problems for 1D linear hyperbolic systems and the fact that
one of the boun\-dary conditions (in our case, this is the last condition in \eqref{12'}) should be considered as an equation for finding the function $\varphi (t,x_2)$.

The number of outgoing characteristics of the mentioned system equals to the number of positive eigenvalues of the matrix $(A_0^+)^{-1}\widetilde{A}_1^+$ plus the number of negative eigenvalues of the matrix $(A_0^-)^{-1}\widetilde{A}_1^-$. It is clear that
\begin{equation}
\tilde{\lambda}^{\pm}_i= \lambda^{\pm}_i -\partial_t\varphi,\quad  i=\overline{1,6},
\label{ev}
\end{equation}
where $\tilde{\lambda}^{\pm}_i$ are the eigenvalues of the matrices $(A_0^{\pm})^{-1}\widetilde{A}_1^{\pm}$, and $\lambda^{\pm}_i$ are the eigenvalues of the characteristic matrices $A_N^{\pm}$ of the hyperbolic systems \eqref{11.1} computed for the normal vector $N=(1, -\partial_2\varphi )$, i.e.,
\[
A_N^{\pm}=\left(A_0(U^{\pm}+\bar{U}^\pm)\right)^{-1}\left( A_1(U^{\pm}+\bar{U}^\pm)-A_2(U^{\pm}+\bar{U}^\pm)\partial_2\varphi  \right)|_{x_1=0}.
\]
As in nonrelativistic MHD \cite{BT,LL}, the eigenvalues of the characteristic matrix depend on the Alfv\'{e}n, slow magnetosonic and fast magnetosonic speeds \cite{An,Lich}. However, as in nonrelativistic MHD, for the 2D case under consideration the Alfv\'{e}n speed ``disappears'', i.e., the eigenvalues are defined only through the slow and fast magnetosonic speeds. Namely, the eigenvalues $\lambda^{\pm}_i$,
\[
\lambda_1^-\leq \ldots \leq\lambda_6^-,\quad \lambda_1^+\leq \ldots \leq\lambda_6^+,
\]
has the form
\begin{equation}
\begin{split}
& \lambda_1^{\pm}=v_N^{\pm}|_{x_1=0}-c_{f_-}^{\pm},\quad \lambda_2^{\pm}=v_N^{\pm}|_{x_1=0}-c_{s_-}^{\pm}, \\ & \lambda_3^{\pm}=\lambda_4^{\pm}=v_N^{\pm}|_{x_1=0},\quad
 \lambda_5^{\pm}=v_N^{\pm}|_{x_1=0}+c_{s_+}^{\pm},\quad \lambda_6^{\pm}=v_N^{\pm}|_{x_1=0}+c_{f_+}^{\pm},
\end{split}
\label{24.1}
\end{equation}
where $c_{s_{\pm}}^{\pm}=c_{s_{\pm}}(U_{|x_1=0}^{\pm},\varphi)$, $c_{f_{\pm}}^{\pm}=c_{f_{\pm}}(U_{|x_1=0}^{\pm},\varphi)$,
and $c_{s_{\pm}}\geq 0$ and  $c_{f_{\pm}}\geq 0$ are the slow and fast magnetosonic speeds respectively \cite{An,Lich} (note that in nonrelativistic MHD  $c_{s_-}=c_{s_+}$ and $c_{f_-}=c_{f_+}$, see \cite{BT,LL}).

Unlike the nonrelativistic setting, there are no relatively simple formulae for the slow and fast magnetosonic speeds  because $c_{s_-}$,  $c_{s_+}$, $c_{f_-}$ and $c_{f_+}$ are the roots of a fourth degree polynomial \cite{An,Lich}. However, for us it is only important that the free coefficient of this polynomial equals
$H_N/{\Gamma^3}$ (see \cite{Anton}). In view of assumption \eqref{mf.1}, the free coefficient does not vanish on the boundary, i.e., the polynomial does not have zero roots. Hence, all the speeds  $c_{s_{\pm}}^{\pm}$ and $c_{f_{\pm}}^{\pm}$ are strictly positive. Then, by virtue of the last condition in \eqref{12'}, it follows from \eqref{ev} and \eqref{24.1} that
\begin{equation}
\tilde{\lambda}_1^{\pm}=-c_{f_-}^{\pm},\quad \tilde{\lambda}_2^{\pm}=-c_{s_-}^{\pm},\quad
\tilde{\lambda}_3^{\pm}=\tilde{\lambda}_4^{\pm}=0,\quad \tilde{\lambda}_5^{\pm}=c^{\pm}_{a_+},\quad \tilde{\lambda}_6^{\pm}=c_{f_+}^{\pm}.
\label{lambda}
\end{equation}
That is, the matrix $(\mathfrak{A}_0)^{-1} \mathfrak{A}_1$ has four positive eigenvalues at each point of the boundary $x_1=0$. This means that systems \eqref{11.1} need five boundary conditions that coincides with the number of conditions in \eqref{12.1}. Since the boundary matrix $\mathfrak{A}_1$ has zero eigenvalues, contact discontinuity is {\it characteristic}. It is also important that the rank of this matrix is constant (it equals eight).

The crucial point in the proof of a priori estimates of the linearized problem for contact discontinuity in nonrelativistic MHD \cite{MTT1,MTT2} was the assumption that the basic state about which the linearization is performed satisfies the condition
\begin{equation}
[\partial_1v]=0
\label{1v}
\end{equation}
(we recall that the jump of the normal derivative is defined in \eqref{norm_jump}). It is important that this condition holds for solutions of the nonlinear problem. This is why this assumption is natural and we can estimate an additional error of the Nash--Moser iteration scheme caused by it (the existence of solutions of the nonlinear problem was being proved in \cite{MTT2} by the Nash--Moser method). In this connection, it is important to prove that the solutions of problem
\eqref{11.1}--\eqref{13.1} for RMHD contact discontinuity satisfy condition \eqref{1v} as well. Moreover, one can show that these solutions also satisfy the condition
\begin{equation}
[\partial_1H_N]=0
\label{1H_N}
\end{equation}

\begin{proposition}
Let problem \eqref{11.1}--\eqref{13.1} (with initial data satisfying \eqref{14.1} and \eqref{15.1}) has a sufficiently smooth solution for which conditions \eqref{mf.1} and \eqref{5.1'} are fulfilled. Then this solution satisfies equalities  \eqref{1v} and \eqref{1H_N}.
\label{p2}
\end{proposition}

\begin{proof}
The first equation of system \eqref{8} reads (we recall that we consider a polytropic gas)
\[
\frac{\Gamma}{\gamma p}\,\frac{{\rm d} p}{{\rm d} t} +v  \cdot \partial_tu +{\rm div}\, u =0.
\]
After straightening the discontinuity and the change of unknowns \eqref{shift} this equation yields the first equations of systems \eqref{11.1}:
\begin{equation}
\begin{split}
\frac{\Gamma^{\pm}}{\gamma (\bar{p} +p^{\pm})}\,& \partial_t p^{\pm}  +v^{\pm}\cdot \partial_tu^{\pm} \\[6pt]   +\frac{1}{\partial_1\Phi^{\pm}}&\bigg\{ \frac{\Gamma^{\pm}}{\gamma (\bar{p} +p^{\pm})}\,(w^{\pm}\cdot \nabla p^{\pm})   -\partial_t\Psi^{\pm}(v^{\pm}\cdot \partial_1u^{\pm})+{\rm div}\,\tilde{u}^{\pm}\bigg\}=0,
\end{split}
\label{41}
\end{equation}
where
\[
w^{\pm}=(v_N^{\pm}-\partial_t\Psi^{\pm},v_2^{\pm}\partial_1\Phi^{\pm}),\quad \tilde{u}^{\pm} =(u_N^{\pm},u_2^{\pm}\partial_1\Phi^{\pm}),\quad u_N^{\pm}=\Gamma^{\pm}v_N^{\pm}.
\]

Restricting \eqref{41} to the boundary $x_1=0$ and taking into account the last boundary condition in \eqref{12'}, we obtain
\[
\frac{\Gamma^{\pm}}{\gamma (\bar{p} +p^{\pm}) }\,\partial^{\pm}_0\,p^\pm +v^{\pm}\cdot \partial_tu^{\pm} \mp v_N^{\pm}(v^{\pm}\cdot \partial_1u^{\pm}) \pm {\rm div}\,\tilde{u}^\pm=0\quad\mbox{at}\ x_1=0,
\]
where $\partial^{\pm}_0:=\partial_t+v_2^\pm\partial_2 $. Summing up the last equations and using the continuity of the pressure and the velocity (cf. \eqref{12'}), we come to the equality
\[
[\partial_1u_N]- v_N^+ (v^{\pm} \cdot [\partial_1u])=0 \quad\mbox{at}\ x_1=0
\]
which, in view of $\partial_1u^{\pm}=\Gamma ^{\pm}(\partial_1 v^{\pm} + (u^{\pm}\cdot \partial_1v^{\pm})u^{\pm})$, is reduced to $\Gamma^+_{|x_1=0}\,[\partial_1v_N]=0$. We thus get
\begin{equation}
[\partial_1v_N]=0.
\label{vn}
\end{equation}

Instead of the third, fourth and fifth equations of systems \eqref{11.1} we can use their equivalent formulation which is the Maxwell equations \eqref{5} in the domains $\Omega^{\pm}(t)$ written in terms of variables \eqref{change}, \eqref{change2}. More precisely, if we multiply the subsystems of systems \eqref{11.1} containing the third, fourth and fifth equations from the left by the nonsingular matrices ${\cal M}^{-1}(U^{\pm})$ (the matrix ${\cal M}$ was written after \eqref{8}) and then pass from the unknowns $u^{\pm}$ to $v^{\pm}$, then we just obtain the mentioned equivalent equations. Passing in them to a jump, we obtain
\[
[\partial_0 H-H_N\partial_1 v  -H_2\partial_2 v+(\partial_1v_N)H]=0,
\]
i.e., by virtue of \eqref{12'} and \eqref{15.1}, one gets
\[
H^+_N[\partial_1 v]=[\partial_1v_N]H^+\quad\mbox{at}\ x_1=0.
\]
In view of equalities \eqref{vn} and conditions \eqref{mf.1}, we get \eqref{1v}. At last, restricting the sum of divergences \eqref{14.1} to the boundary, we obtain equality \eqref{1H_N}.
\end{proof}

\section{Main result and discussion}

We are now in a position to state the main result of the present paper that is the local-in-time existence and uniqueness theorem for problem \eqref{11.1}--\eqref{13.1}. Clearly, this theorem implies a corresponding theorem for the original free boundary problem  \eqref{17}--\eqref{19}.

\begin{theorem} Let $m\in\mathbb{N}$ and $m\geq 6$. Suppose the initial data \eqref{13.1}, with
\[
\left((U^+_0,U^-_0),\varphi_0\right)\in H^{m+17/2}(\Omega)\times H^{m+17/2}(\partial\Omega)
\]
satisfy the hyperbolicity condition \eqref{5.1'}, requirement  \eqref{5.1"}  and the divergence constraints \eqref{14} for all $x\in\Omega$. Let the initial data also satisfy requirement \eqref{mf.1}, the Rayleigh--Taylor sign condition \eqref{RT1} and the boundary constraint \eqref{15} for all $x\in \partial\Omega$. Assume also that the initial data are compatible up to order $m+8$ in the sense analogous to that in \cite{MTT2}. Then there exists a sufficiently short time $T>0$ such that problem \eqref{11.1}--\eqref{13.1} has a unique solution
\[
\left((U^+,U^-),\varphi\right)\in H^{m}([0,T]\times\Omega)\times H^{m}([0,T]\times\partial\Omega).
\]
\label{t1}
\end{theorem}
In the theorem above we use the standard notation $H^m$ for the Sobolev spaces $W_2^m$, which are Hilbert spaces.

The boundary conditions \eqref{12'} coincide with those for the contact discontinuity in nonrelativistic MHD  \cite{MTT1,MTT2}. As was noted in  \cite{MTT1},
for such boundary conditions the symbol associated with the free boundary is not elliptic. This means that the boundary conditions are not resolvable for $\nabla_{t,x}\varphi=(\partial_t\varphi ,\partial_2\varphi )$. This implies that the corresponding constant coefficient linearized problem does not satisfy the uniform Lopatinski condition, i.e., the Lopatinski conditions holds only in a weak sense (see, e.g., \cite{BT}). This also implies a so-called loss of derivatives from the initial data and the source terms in a priori estimates of solutions of the linearized problem (for both constant and variables coefficients). In such cases, often but not always the existence theorem for the nonlinear problem is proved by Nash--Moser iterations as, for example, in \cite{MTT2}. Then the initial data are assumed to be smother than the solutions (cf. Theorem \ref{t1}).

It is worth noting that, as in nonrelativistic MHD, the Rayleigh--Taylor sign condition \eqref{RT1} does not contain the magnetic field, unlike, for example, the plasma-vacuum interface problem for which such a condition is written in \cite{Tjde} for the total pressure $q$ (in nonrelativistic MHD $q=p+|H|^2/2$). That is, for contact discontinuities the Rayleigh--Taylor sign condition appears in its classical (hydrodynamical) form.

It is natural expect that the fulfillment of the Rayleigh--Taylor sign condition at the first moment guarantees the local existence of the contact discontinuity in the general 3D case. However, the proof of the corresponding existence and uniqueness theorem is still an open problem even for the technically simpler nonrelativistic case (see discussion in \cite{MTT1}).

The natural question is whether the Rayleigh--Taylor sign condition (besides other assumptions on the initial data in Theorem \ref{t1} and the analogous theorem in \cite{MTT2}) is not only sufficient but also necessary for the well-posedness of the free boundary problem. As was noted in \cite{Tcpaa}, Rayleigh--Taylor-type instability, in particular, means the Hadamard-type ill-posedness of the corresponding frozen coefficient linearized problem (see also the interesting survey \cite{JosSaut}). Formally, there is no general result saying that such kind of ill-posedness implies also the ill-posedness of the original nonlinear free boundary problem. However, we know, at least, two examples of problems for which there is a rigorous proof that Rayleigh--Taylor instability implies the ill-posedness of the nonlinear problem. These are the free boundary problems for an ideal incompressible \cite{Ebin} and compressible liquid \cite{GuoTice}.

For the contact discontinuity, due to big technical difficulties the question of constructing the Hadamard-type ill-posedness example for the frozen coefficient linearized problem under the violation of the Rayleigh--Taylor sign condition is still an open problem even for nonrelativistic MHD. Technical difficulties are caused, first off all, by the fact that the magnetic field on the the contact discontinuity is nowhere parallel to its surface. For instance, for the plasma-vacuum interface problem, for which the magnetic field on the boundary is just parallel to it, an ill-posedness example was managed to be constructed in \cite{Tcpaa} for the corresponding frozen coefficient linearized problem under the simultaneous failure of the Rayleigh--Taylor sign condition for the total pressure $q$ and the noncollinearity condition for the plasma and vacuum magnetic fields on the boundary (see \cite{ST,Tjde,Tcpaa}).

Regarding strong discontinuities in RMHD, as in nonrelativistic MHD \cite{BT,LL}, besides contact discontinuities under consideration there are shock waves ($\mathfrak{j}\neq 0$, $[\rho ]\neq 0$, see \eqref{10}--\eqref{15}), current-vortex sheets ($\mathfrak{j}= 0$, $H_n=0$) and Alfv\'{e}n discontinuities ($\mathfrak{j}\neq 0$, $[\rho ]= 0$). For shock waves in RMHD, unlike the nonrelativistic case (see  \cite{BT,T}), there is only the result in  \cite{T_qam} where the domains of uniform stability, neutral stability and violent instability were found for the fast parallel shock waves. For current-vortex sheets in RMHD, using the ideas of \cite{T05,T09}, a sufficient ``stability'' condition whose fulfillment at each point of the initial discontinuity (together with compatibility conditions, etc.) implies a local-in-time existence and uniqueness theorem was found in \cite{FT}. Returning to contact discontinuities in RMHD, we note that Theorem \ref{t1} is totally analogous to the existence and uniqueness theorem in \cite{MTT2} for the nonrelativistic case.

\section{Basic a priori estimate for the linearized \\ problem}

For the proof of Theorem  \ref{t1} for the nonlinear problem  \eqref{11.1}--\eqref{13.1} by Nash--Moser iterations (see, e.g., \cite{MTT2,ST,T09,Tcpam}) it is necessary to perform a detailed analysis of the corresponding variable coefficient linearized problem. The linearized problem for the RMHD contact discontinuity has no principal differences from the corresponding problem in the nonrelativistic case. However, we have to discuss some peculiarities of the relativistic case for certain points in the proof of well-posedness for the  linearized problem. This is why for readers' convenience we write down here this linearized problem.

Let
\[
\begin{split}
& \Omega_T:= (-\infty, T]\times\Omega,\quad \partial\Omega_T:=(-\infty ,T]\times\partial\Omega ,\\
 & \Omega_T^+:= [0, T]\times\Omega,\quad \partial\Omega_T^+:=[0 ,T]\times\partial\Omega .
\end{split}
\]
We consider the {\it basic state}
\begin{equation}
(\widehat{U}^+(t,x ),\widehat{U}^-(t,x ),\hat{\varphi}(t,{x}'))
\label{a21}
\end{equation}
upon which we perform linearization. Let
\[
\widehat{U}^{\pm}=(\hat{p}^{\pm},\hat{u}^{\pm},\widehat{H}^{\pm},\widehat{S}^{\pm})
\]
and $\hat{\varphi}$ be given sufficiently smooth functions, with
\begin{equation}
\|\widehat{U}^+\|_{W^2_{\infty}(\Omega_T)}+
\|\widehat{U}^-\|_{W^2_{\infty}(\Omega_T)}
+\|\hat{\varphi}\|_{W^2_{\infty}(\partial\Omega_T)} \leq K,
\label{a22}
\end{equation}
where $K>0$ is a constant. We will use the notations
\[
\widehat{\Phi}^{\pm}(t,x )=\pm x_1 +\widehat{\Psi}^{\pm}(t,x ),\quad
\widehat{\Psi}^{\pm}(t,x )=\chi(\pm x_1)\hat{\varphi}(t,x').
\]
That is, all the ``hat'' values are determined like corresponding values for $(U^+,U^-, \varphi)$, for example,
\[
\hat{v}_{N}^{\pm}=\hat{v}_1^{\pm}- \hat{v}_2^{\pm}\partial_2\widehat{\Psi}^{\pm},\quad
\widehat{u}^{\pm}=\widehat{\Gamma}^{\pm}\hat{v}^{\pm},\quad \widehat{\Gamma}^{\pm}=\sqrt{1+|\hat{u}^{\pm}|^2},\quad
\widehat{H}^{\pm}_{\tau}= \widehat H_1^\pm \partial_2\widehat\Psi^{\pm}+\widehat H^{\pm}_2.
\]
Moreover, without loss of generality we assume that $\|\hat{\varphi}\|_{L_{\infty}(\partial\Omega_T)}<1$ (cf. \eqref{fi}). This implies $\partial_1\widehat{\Phi}^+\geq 1/2$ and $\partial_1\widehat{\Phi}^-\leq -  1/2$.

We assume that the basic state satisfies the ``relaxed'' hyperbolicity condition \eqref{5.1'},
\begin{equation}
\hat{p}^{\pm} \geq -\bar{p}/2\quad \mbox{in}\ \Omega_T^+,  \label{a5}
\end{equation}
the ``relaxed'' requirement \eqref{5.1"} that the speed of light is unattainable,
\begin{equation}
1-|\hat{v}^{\pm}|\geq \nu /2 >0\quad \mbox{in}\ \Omega_T^+,
\label{ls}
\end{equation}
the boundary conditions \eqref{12'} together with the boundary constraint \eqref{15},
\begin{equation}
[\hat{p}]=0,\quad [\hat{v}]=0,\quad [\widehat{H}]=0, \quad \partial_t\hat{\varphi}=\hat{v}_{N}^+|_{x_1=0}\quad \mbox{on}\ \partial\Omega_T,
\label{a12'}
\end{equation}
the ``relaxed'' requirement \eqref{mf.1},
\begin{equation}
|\widehat{H}_N^{\pm}|_{x_1=0}|\geq {\kappa}/2 >0\quad \mbox{on}\ \partial\Omega_T^+,
\label{cdass}
\end{equation}
and the jump conditions \eqref{1v} and \eqref{1H_N},
\begin{equation}
[\partial_1\hat{v}]=0,\quad [\partial_1\widehat{H}_N]=0\quad\mbox{on}\ \partial\Omega_T.
\label{jc1'}
\end{equation}

The linearization of equations  \eqref{11.1} and \eqref{12.1} about the basic state \eqref{a21} reads:
\[
\mathbb{L}'(\widehat{U}^{\pm},\widehat{\Psi}^{\pm})(\delta U^{\pm},\delta\Psi^{\pm}):=
\frac{d}{d\varepsilon}\mathbb{L}(U_{\varepsilon}^{\pm},\Psi_{\varepsilon}^{\pm})|_{\varepsilon =0}={f}^{\pm}
\quad \mbox{in}\ \Omega_T,
\]
\[
\mathbb{B}'(\widehat{U}^+,\widehat{U}^-,\hat{\varphi})(\delta U^+,\delta U^-,\delta \varphi):=
\frac{d}{d\varepsilon}\mathbb{B}(U_{\varepsilon}^+,U_{\varepsilon}^-,\varphi_{\varepsilon})|_{\varepsilon =0}={g}
\quad \mbox{on}\ \partial\Omega_T
\]
where $U_{\varepsilon}^{\pm}=\widehat{U}^{\pm}+ \varepsilon\,\delta U^{\pm}$,
$\varphi_{\varepsilon}=\hat{\varphi}+ \varepsilon\,\delta \varphi$ and
\[
\Psi_{\varepsilon}^{\pm}(t,{x} ):=\chi (\pm x_1)\varphi _{\varepsilon}(t,{x}'),\quad
\Phi_{\varepsilon}^{\pm}(t,{x} ):=\pm x_1+\Psi_{\varepsilon}^{\pm}(t,{x} ),
\]
\[
\delta\Psi^{\pm}(t,{x} ):=\chi (\pm x_1)\delta \varphi (t,{x} ).
\]
Here we introduce the source terms
\[
{f}^{\pm}(t,{x} )=(f_1^{\pm}(t,{x} ),\ldots ,f_6^{\pm}(t,{x} ))\quad \mbox{and}\quad {g}(t,{x}' )=(g_1(t,{x}' ),\ldots ,g_5(t,{x}' ))
\]
to make the interior equations and the boundary conditions inhomogeneous.

We easily compute the exact form of the linearized equations (below we drop $\delta$ in the notations of perturbations):
\[
\begin{split}
\mathbb{L}'& (\widehat{{U}}^{\pm}, \widehat{\Psi}^{\pm})({U}^{\pm},\Psi^{\pm}) \\ &
=
L (\widehat{{U}}^{\pm},\widehat{\Psi}^{\pm}){U}^{\pm} +{\cal C}(\widehat{{U}}^{\pm},\widehat{\Psi}^{\pm})
{U}^{\pm}  -  \bigl\{L(\widehat{{U}}^{\pm},\widehat{\Psi}^{\pm})\Psi^{\pm}\bigr\}\frac{\partial_1\widehat{U}{\pm}}{\partial_1\widehat{\Phi}^{\pm}},
\end{split}
\]
\[
\mathbb{B}'(\widehat{{U}}^+,\widehat{{U}}^-,\hat{\varphi})({U}^+,{U}^-,\varphi)=
\left(
\begin{array}{c}
p^+-p^-\\[3pt]
v^+-v^-\\[3pt]
H_{\tau}^+-H_{\tau}^-\\[3pt]
\partial_t\varphi +\hat{v}_2^+\partial_2\varphi  -v_{N}^+
\end{array}
\right),
\]
where $v_{N}^{\pm}=v_1^{\pm}-v_2^{\pm}\partial_2\widehat{\Psi}^\pm$, ${H}^{\pm}_{\tau}= H_1^\pm \partial_{2}\widehat{\Psi}^{\pm}+H^{\pm}_{2}$ and the matrix
${\cal C}(\widehat{{U}}^{\pm},\widehat{\Psi}^{\pm})$ is determined as
\[
\begin{split}
{\cal C}(\widehat{{U}}^{\pm},\widehat{\Psi}^{\pm}){Y}
= (&{Y} , \nabla_yA_0(\widehat{{U}}^{\pm} +\bar{U}^{\pm}))\partial_t\widehat{{U}}^{\pm}
+({Y} ,\nabla_y\widetilde{A}_1(\widehat{U}^{\pm}+\bar{U}^{\pm},\widehat{\Psi}^{\pm}))\partial_1\widehat{{U}}^{\pm}\\ &
+ ({Y} ,\nabla_yA_2(\widehat{{U}}^{\pm} +\bar{U}^{\pm}))\partial_2\widehat{{U}}^{\pm},
\end{split}
\]
\[
({Y} ,\nabla_y A(\widehat{{U}}^{\pm}+\bar{U}^{\pm})):=\sum_{i=1}^6y_i\left.\left(\frac{\partial A ({Y} )}{
\partial y_i}\right|_{{Y} =\widehat{{U}}^{\pm}+\bar{U}^{\pm}}\right),\quad {Y} =(y_1,\ldots ,y_6).
\]
We also note that $v^{\pm}$ are expressed through $u^{\pm}$ as follows:
\begin{equation}
v^{\pm}=\frac{1}{\widehat{\Gamma}^{\pm}}\left(u^{\pm}-(\hat{v}^{\pm}\cdot u^{\pm})\hat{v}^{\pm}\right).
\label{vu}
\end{equation}
We recall that $\delta$ was dropped in the notations of perturbations, i.e., here $v^{\pm}:=\delta v^{\pm}$  and $u^{\pm}:=\delta u^{\pm}$.

The differential operator $\mathbb{L}'(\widehat{U}^{\pm},\widehat{\Psi}^{\pm})$ is a first order operator in
$\Psi^{\pm}$. Following \cite{Al}, we overcome this potential difficulty by introducing the ``good unknown''
\begin{equation}
\dot{U}^\pm:=U^\pm -\frac{\Psi^\pm}{\partial_1\widehat{\Phi}^\pm}\,\partial_1\widehat{U}^\pm
\label{b23}
\end{equation}
Omitting detailed calculations, we rewrite the linearized interior equations in terms of $\dot{U}^+$ and $\dot{U}^-$:
\begin{equation}
L(\widehat{U}^{\pm},\widehat{\Psi}^{\pm})\dot{U}^{\pm} +{\cal C}(\widehat{U}^{\pm},\widehat{\Psi}^{\pm})
\dot{U}^{\pm} - \frac{\Psi^{\pm}}{\partial_1\widehat{\Phi}^{\pm}}\,\partial_1\bigl\{\mathbb{L}
(\widehat{U}^{\pm},\widehat{\Psi}^{\pm})\bigr\}={f}^{\pm}.
\label{b24}
\end{equation}
As in \cite{Al,MTT1,MTT2,ST,T09,Tcpam}, we drop the zeroth-order terms in $\Psi^+$  and $\Psi^-$ in \eqref{b24} and consider the effective linear operators
\[
\begin{split}
\mathbb{L}'_e(\widehat{U}^{\pm},\widehat{\Psi}^{\pm})\dot{U}^{\pm} :=& L(\widehat{U}^{\pm},\widehat{\Psi}^{\pm}) \dot{U}^{\pm}+\mathcal{C}(\widehat{U}^{\pm},\widehat{\Psi}^{\pm})\dot{U}^{\pm}
\\
= & A_0(\widehat{U}^{\pm}+\bar{U}^{\pm})\partial_t\dot{U}^{\pm}
+\widetilde{A}_1(\widehat{U}^{\pm}+\bar{U}^{\pm},\widehat{\Psi}^{\pm})
\partial_1\dot{U}^{\pm}\\
 &+A_2(\widehat{U}^{\pm}+\bar{U}^{\pm})\partial_2\dot{U}^{\pm}+\mathcal{C}(\widehat{U}^{\pm},\widehat{\Psi}^{\pm})\dot{U}^{\pm}.
\end{split}
\]
In the subsequent nonlinear analysis the dropped terms in \eqref{b24} should be considered as additional error terms of the Nash--Moser iterations step.

Regarding the boundary differential operator $\mathbb{B}'$, in terms of unknowns \eqref{b23} it reads
\[
\mathbb{B}'_e(\widehat{U},\hat{\varphi})(\dot{U},\varphi ):=
\mathbb{B}'(\widehat{U},\hat{\varphi})(U^+,U^-,\varphi )=
\left(
\begin{array}{c}
\dot{p}^+-\dot{p}^- + \varphi [\partial_1\hat{p}]\\[3pt]
\dot{v}_1^+-\dot{v}_1^-\\[3pt]
\dot{v}_2^+-\dot{v}_2^-\\[3pt]
\dot{H}_{\tau}^+-\dot{H}_{\tau}^-+ \varphi [\partial_1\widehat{H}_{\tau}]\\[3pt]
\partial_t\varphi +\hat{v}_2^+\partial_2\varphi -\dot{v}_{N}^+ - \varphi \partial_1\hat{v}_N^+
\end{array}
\right),
\]
where
\[
\widehat{U}=(\widehat{U}^+,\widehat{U}^-),\quad \dot{U}=(\dot{U}^+,\dot{U}^-),
\quad \dot{v}_{N}^{\pm}=
\dot{v}_1^{\pm}-\dot{v}_2^{\pm}\partial_2\widehat{\Psi}^{\pm}, \quad
\dot{H}_{\tau}^{\pm}=\dot{H}_1^{\pm}\partial_2\widehat{\Psi}^{\pm} +\dot{H}_2^{\pm}.
\]
We used here the important condition $[\partial_1\hat{v}]=0$. Note that \eqref{vu} and \eqref{b23} imply
\begin{equation}
\dot{v}^{\pm}=\frac{1}{\widehat{\Gamma}^{\pm}}\left(\dot{u}^{\pm}-(\hat{v}^{\pm}\cdot \dot{u}^{\pm})\hat{v}^{\pm}\right).
\label{vu'}
\end{equation}

Introducing the notation
\[
\mathbb{L}'_e(\widehat{U},{\widehat{\Psi}})\dot{U}:=
\left(
\begin{array}{c}
\mathbb{L}'_e(\widehat{U}^+,\widehat{\Psi}^+)\dot{U}^+\\[3pt]
\mathbb{L}'_e(\widehat{U}^-,\widehat{\Psi}^-)\dot{U}^-
\end{array}
\right),
\]
with $\widehat{\Psi}=(\widehat{\Psi}^+,\widehat{\Psi}^-)$, we write down the linear problem for $(\dot{U},\varphi)$:
\begin{align}
\mathbb{L}'_e(\widehat{U},\widehat{{\Psi}})\dot{U}={f}\quad &\mbox{in}\ \Omega_T, \label{28}\\[3pt]
\mathbb{B}'_e(\widehat{U},\hat{\varphi})(\dot{U},\varphi)={g}\quad &\mbox{on}\ \partial\Omega_T,\label{29}\\[3pt]
(\dot{U},\varphi)=0\quad &\mbox{for}\ t<0,\label{30}
\end{align}
where $f=(f^+,f^-)$. Here we suppose $f$ and $g$ vanish in the past (for $t<0$) and consider the case of zero initial data. The case of nonzero initial data can be postponed to the nonlinear analysis (construction of a so-called approximate solution; see, e.g., \cite{Tcpam}).

\begin{theorem}
Let assumptions \eqref{a22}--\eqref{jc1'} be fulfilled for the basic state \eqref{a21}. Let also the basic state satisfies  the Rayleigh--Taylor sign condition
\begin{equation}
[\partial_1\hat{p} ]\geq \epsilon /2 >0\quad \mbox{on}\ \partial\Omega_T^+,
\label{RTL}
\end{equation}
where $[\partial_1\hat{p} ]=\partial_1\hat{p}^+_{|x_1=0} +\partial_1\hat{p}^-_{|x_1=0}$ (see \eqref{norm_jump}). Then, for all $f \in H^2(\Omega_T)$ and $g\in H^{2}(\partial\Omega_T)$ which vanish in the past, problem \eqref{28}--\eqref{30} has a unique solution $(\dot{U},\varphi )\in H^1(\Omega_T)\times H^1(\partial\Omega_T)$. Moreover, this solution obeys the a priori estimate
\begin{equation}
\|\dot{U} \|_{H^{1}(\Omega_T)}+\|\varphi\|_{H^1(\partial\Omega_T)} \leq C\left\{\|f \|_{H^{2}(\Omega_T)}+ \|g\|_{H^{2}(\partial\Omega_T)}\right\},
\label{main_estL}
\end{equation}
where $C=C(K,\bar{p},{\kappa},\epsilon,T)>0$ is a constant independent of the data $f$ and $g$.
\label{t1L}
\end{theorem}

Referring the reader to the detailed proof of the corresponding theorem in \cite{MTT1} for nonrelativistic MHD, we briefly comment here the proof of Theorem \ref{t1L} by focusing on places which need special attention in the context of the relativistic case. We will also refer to \cite{MTT2} where a priori estimates of the linearized problem were being deduced in higher-order norms, but the process of deriving these estimates differs a little from that of getting the basic $H^1$ a priori estimate in \cite{MTT1}.

We note that in \cite{MTT1} in the right-hand side of the counterpart of the a priori estimate \eqref{main_estL} there appear the norms $\|f \|_{H^{1}(\Omega_T)}$ and $\|g\|_{H^{3/2}(\partial\Omega_T)}$. First, we have especially roughened estimate \eqref{main_estL} by writing $\|g\|_{H^{2}(\partial\Omega_T)}$ instead of the norm $\|g\|_{H^{3/2}(\partial\Omega_T)}$ because the gain of one half derivative is not used in the proof of the existence theorem for the nonlinear problem. Second, regarding the source term  $f$, the peculiarity of our relativistic case is such that estimate \eqref{main_estL} is worse a little bit in the sense of a bigger requirements on $f$. However, this does not influence on the final result of Theorem \ref{t1} for the nonlinear problem. The regularity of the initial data and solutions in Theorem \ref{t1} is the same as that in the analogous theorem in \cite{MTT2}. As was noted in \cite{Tcpam}, minor restrictions on the smoothness of the source term $f$ in comparison with $g$ play no role in the proof of convergence of the Nash--Moser iterations. That is, one finally uses an a priori estimate with the same index of Sobolev space for  $f$ and $g$.

By standard argument we first reduce the boundary conditions to homogeneous ones (see \cite{MTT1,MTT2}). However, for our relativistic case we have to modify a little bit the standard procedure. First, as in \cite{MTT1,MTT2}, we introduce a vector-function  $\widetilde{U}=(\widetilde{U}^+,\widetilde{U}^-)$ vanishing in the past and satisfying the boundary conditions \eqref{29} for $\varphi =0$. Then, the new unknown
\begin{equation}
\dot{U}^{\natural}=(U^{+\natural},U^{-\natural})=\dot{U}-\widetilde{U}, \label{a87'}
\end{equation}
for which we have the estimate
\begin{equation}
\|\widetilde{U} \|_{H^{s+1}(\Omega_T)}\leq C\|g \|_{H^{s+1/2}(\partial\Omega_T)}
\label{tildU}
\end{equation}
(in this section $s=0$ or $s=1$), satisfies problem \eqref{28}--\eqref{30} with $f=\tilde{f} =(\tilde{f}^+,\tilde{f}^-)$, where
\begin{equation}
\tilde{f}^\pm =f^\pm-\mathbb{L}'_e(\widehat{U}^{\pm},\widehat{\Psi}^{\pm})\widetilde{U}^{\pm}.
\label{a87''}
\end{equation}
Here and below $C$ is a positive constant that can change from line to line, and it does not depend on the source terms but may depend on other constants, in particular, in \eqref{a22} it depends on the constant $K$.

Choosing $\widetilde{U}$ such that it satisfies the boundary conditions \eqref{29}, we have the requirement $[\widetilde{H}_{\tau}]=g_4$ only for the tangential components of the vectors $\widetilde{H}^{\pm}$ whereas for their normal components $\widetilde{H}_N^{\pm}$ we still have some freedom. We require that the normal components satisfy the condition $\big[\widetilde{H}_N\big]=g_6$, where  $g_6$ is a solution of the equation
\begin{equation}
\label{tHN}
\partial_tg_6 + \partial_2(\hat{v}_2^+g_6) =\big[f_H\cdot \widehat{N}\big]\quad \mbox{on}\ \partial\Omega_T
\end{equation}
(we below discuss why exactly this equation was taken), where  $\widehat{N}^{\pm}=(1,-\partial_2\widehat{\Psi}^{\pm})$, and ${f}^{\pm}_{H}=({f}^{1\pm}_{H},{f}^{2\pm}_{H})$ are right-hand parts in the equations for $\dot{H}^{\pm}$ following from \eqref{28} (we below write down these equations $\partial_t\dot{H}^{\pm}+\ldots = {f}^{\pm}_{H}$ up to lower-order terms). The right-hand parts ${f}^{\pm}_{H}$ being determined through $f^{\pm}$ and the basic state \eqref{a21} can be written down if necessary, but their concrete form is of no interest. Instead of \eqref{tildU} we now have the estimate
\begin{equation}
\|\widetilde{U} \|_{H^{s+1}(\Omega_T)}\leq C\|(g, g_6) \|_{H^{s+1/2}(\partial\Omega_T)}
\label{tildU"}
\end{equation}
(note that $\|\widetilde{H}_N^{\pm}\|_{H^{s+1}(\Omega_T)}\leq C\|g_6 \|_{H^{s+1/2}(\partial\Omega_T)}$).

Using the trace theorem for the functions ${f}^{\pm}_{H}$, from \eqref{tHN} we deduce the estimate
\[
\|g_6 \|_{H^{3/2}(\partial\Omega_T)}\leq C\|f |_{x_1=0}\|_{H^{3/2}(\partial\Omega_T)}\leq C\|f \|_{H^{2}(\Omega_T)}.
\]
Then, for $\widetilde{H}_N^{\pm}$ we have:
\[
\|\widetilde{H}_N^{\pm}\|_{H^{2}(\Omega_T)}\leq C\|g_6 \|_{H^{3/2}(\partial\Omega_T)}\leq C\|f \|_{H^{2}(\Omega_T)}.
\]
In view of \eqref{tildU"} (for $s=1$), we finally obtain the following estimate for $\widetilde{U}$:
\begin{equation}
\|\widetilde{U} \|_{H^{2}(\Omega_T)}\leq C\left(\|g \|_{H^{3/2}(\partial\Omega_T)} + \|f \|_{H^{2}(\Omega_T)}\right).
\label{tildU'}
\end{equation}
Clearly, by virtue of \eqref{tildU'}, one gets the estimate
\begin{equation}\label{tf}
\|\tilde{f} \|_{H^{1}(\Omega_T)}\leq C\left(\|g \|_{H^{2}(\partial\Omega_T)} + \|f \|_{H^{2}(\Omega_T)}\right)
\end{equation}
for the new right-hand part $\tilde{f}$ (see \eqref{a87''}). We have roughened this estimate by writing $\|g \|_{H^{2}(\partial\Omega_T)}$ instead of the norm $\|g \|_{H^{3/2}(\partial\Omega_T)}$.

Regarding equation \eqref{tHN}, it was shown in \cite{MTT1} that the jump $\big[\dot{H}_N\big]$ satisfied exactly this equation (for $g=0$). Then, after the shift of unknowns  \eqref{a87'} we get the equation
\[
\partial_t\big[\dot{H}_N^{\natural}\big] + \partial_2\big(\hat{v}_2^+\big[\dot{H}_N^{\natural}\big]\big) =0\quad \mbox{on}\ \partial\Omega_T
\]
implying $\big[\dot{H}_N^{\natural}\big]=0$ for zero initial data.

Dropping ${\natural}$ and tildes after the shift of unknowns \eqref{a87'}, we obtain problem \eqref{28}--\eqref{30} with $g=0$ whose solutions satisfy the additional condition
\begin{equation}\label{HN"}
\big[\dot{H}_N\big]=0
\end{equation}
on the boundary. For this problem with homogeneous boundary conditions and the additional condition \eqref{HN"} we should deduce the estimate
\begin{equation}
\|\dot{U} \|_{H^{1}(\Omega_T)}+\|\varphi\|_{H^1(\partial\Omega_T)} \leq C\|f \|_{H^{1}(\Omega_T)}.
\label{main_estL'}
\end{equation}
Then, taking into account inequalities \eqref{tildU'} and \eqref{tf},  from \eqref{main_estL'} we  get the desired a priori estimate \eqref{main_estL}.

For deriving the a priori estimate \eqref{main_estL'}, as in \cite{MTT1}, we first get the auxiliary estimate
\begin{equation}\label{d1U}
\|\partial_1\dot{U}(t)\|^2_{L^2(\Omega)}\leq C\bigg\{\|f\|^2_{H^1(\Omega_T)} +\nt \dot{U}(t)\nt^2_{{\rm tan},1}+ \int\limits_0^t\nt\dot{U}(s)\nt^2_{H^1(\Omega)}{\rm d}s \bigg\},
\end{equation}
where
\[
\begin{split}
& \nt \dot{U}(t)\nt^2_{{\rm tan},1}:=\|\dot{U}(t)\|^2_{L^2(\Omega)}+\|\partial_t\dot{U}(t)\|^2_{L^2(\Omega)}+\|\partial_2\dot{U}(t)\|^2_{L^2(\Omega)},\\
& \nt\dot{U}(s)\nt^2_{H^1(\Omega)}:=\nt \dot{U}(t)\nt^2_{{\rm tan},1}+\|\partial_1\dot{U}(t)\|^2_{L^2(\Omega)}.
\end{split}
\]
Here and below auxiliary inequalities are called estimates (they are not estimates in the usual sense). For getting estimate \eqref{d1U} we should use assumption \eqref{cdass} and take into account the structure of the boundary matrix $\mathfrak{A}_1$ (see \eqref{A1}) written down for the basic state \eqref{a21}. Unlike the much more simple nonrelativistic case, it is not easy to write down the boundary matrix explicitly, as was done in \cite{MTT1,MTT2}. But, fortunately we do not need to do this.

As a matter of fact, the boundary matrices $\pm\widetilde{A}_1(\widehat{U}^{\pm},\widehat{\Psi}^{\pm})|_{x_1=0}$ have two zero eigenvalues (see \eqref{lambda}). We first have to understand what unknowns (we will call them characteristic) correspond to these zero eigenvalues. First of all, $\dot{S}^{\pm}$ are such characteristic unknowns (exactly as in \cite{MTT1,MTT2}). Indeed, since the last equation in \eqref{8} is ${\rm d}S/{\rm d}t=0$, its linearization, in view of the preliminary change of unknowns \eqref{change}, \eqref{change2} and the subsequent change of unknowns \eqref{b23}, yields the sixth and eleventh equations of system \eqref{28}:
\begin{equation}\label{S}
\partial_t\dot{S}^{\pm}+\frac{1}{\partial_1\widehat{\Phi}^{\pm}}(\hat{w}^{\pm}\cdot\nabla\dot{S}^{\pm})+\mbox{l.o.t.}=f_6^{\pm},
\end{equation}
where $\hat{w}^{\pm}$ are the vectors  $w^{\pm}$ appearing in \eqref{41} and written down for the basic state  \eqref{a21}. Here and below l.o.t. denotes lower-order terms whose concrete form is of no interest. In view of the last assumption in \eqref{a12'}, we have $\hat{w}^{\pm}_1|_{x_1=0}=0$. That is, $\dot{S}^{\pm}$ are indeed characteristic unknowns.

Regarding two more characteristic unknowns, as in nonrelativistic MHD \cite{MTT1,MTT2}, they are $\dot{H}_N^{\pm}=
\dot{H}_1^{\pm}-\dot{H}_2^{\pm}\partial_2\widehat{\Psi}^{\pm}$.  Indeed, taking into account the preliminary change of unknowns \eqref{change}, \eqref{change2} and the subsequent change of unknowns \eqref{b23}, the linearization of system \eqref{5} gives the equations
\begin{equation}
\partial_t\dot{H}^{\pm}+\frac{1}{\partial_1\widehat{\Phi}^{\pm}}\left\{ (\hat{w}^{\pm} \cdot\nabla )
\dot{H}^{\pm} - (\hat{\mathfrak{h}}^{\pm} \cdot\nabla ) \dot{v}^{\pm} + \widehat{H}^{\pm}{\rm div}\,\tilde{\dot{v}}^{\pm}
\right\} +\mbox{l.o.t.}={f}^{\pm}_{H}
\label{aA3}
\end{equation}
which follow from \eqref{28}, where $\tilde{\dot{v}}^{\pm} =(\dot{v}_N^{\pm},\dot{v}_2^{\pm}\partial_1\widehat{\Phi}^{\pm})$ (see also \eqref{vu'}). At the same time, it follows from \eqref{aA3} that
\begin{equation}
\partial_t\dot{H}_N^{\pm}+\frac{1}{\partial_1\widehat{\Phi}^{\pm}}( \hat{w}^{\pm} \cdot\nabla \dot{H}_N^{\pm}) - \widehat{H}_2^{\pm} \partial_2\dot{v}_N^{\pm} +\widehat{H}_N^{\pm}\partial_2\dot{v}_2^{\pm} +\mbox{l.o.t.}=({f}^{\pm}_{H}\cdot \widehat{N}^{\pm}).
\label{dHN}
\end{equation}
In view of $\hat{w}^{\pm}_1|_{x_1=0}=0$, we are not able to resolve the last equations for $\partial_1\dot{H}_N^{\pm}$ at each point of $\Omega_T$, i.e., $\dot{H}_N^{\pm}$ are indeed characteristic unknowns.

The unknowns $\dot{p}^{\pm}$, $\dot{u}^{\pm}$ and $\dot{H}_2^{\pm}$ (or $\dot{H}_{\tau}^{\pm}$) are thus noncharacteristic.
It follows from \eqref{cdass} that the speeds $c_{s_{\pm}}^{\pm}$ and $c_{f_{\pm}}^{\pm}$ computed on the basic state are strictly positive (see \eqref{lambda}).
In view of the smoothness assumption for the basic state \eqref{a21}, we can resolve system  \eqref{28} for the $x_1$--derivatives of these noncharacteristic unknowns in a small neighbourhood of the boundary $x_1=0$. Following then standard arguments from \cite{MTT1,MTT2}, we finally get the estimate
\begin{equation}\label{d1Un}
\|\partial_1\dot{U}_n(t)\|^2_{L^2(\Omega)}\leq C\bigg\{\|f\|^2_{H^1(\Omega_T)} +\nt \dot{U}(t)\nt^2_{{\rm tan},1}+ \int\limits_0^t\nt\dot{U}(s)\nt^2_{H^1(\Omega)}{\rm d}s \bigg\}
\end{equation}
in the whole domain $\Omega$, where  $\dot{U}_n=(\dot{U}_n^+,\dot{U}_n^-)$ and $\dot{U}_n^{\pm}=(\dot{p}^{\pm}, \dot{u}^{\pm},\dot{H}_2^{\pm})$.

Following then \cite{MTT1,MTT2}, for estimating the $x_1$--derivatives of the noncharacteristic unknowns $\dot{H}_N^{\pm}$ and $\dot{S}^{\pm}$ we use equations \eqref{S} as well as the equations for the linearized divergences of the magnetic fields
\[
\xi^{\pm} ={\rm div}\,\dot{\mathfrak{h}}^{\pm}= \partial_1\dot{H}_N^{\pm} +\partial_2(\dot{H}_2^{\pm}\partial_1\widehat{\Phi}^{\pm})
\]
obtained by applying the divergence operator to the systems for $\dot{\mathfrak{h}}^{\pm}$, which are implied by systems \eqref{aA3}:
\begin{equation}\label{xi}
\partial_t \left(\frac{\xi^{\pm}}{\partial_1\widehat{\Phi}^{\pm}}\right)+ \frac{1}{\partial_1\widehat{\Phi}^{\pm}} \left(\hat{w}^{\pm} \cdot\nabla \left(\frac{\xi^{\pm}}{\partial_1\widehat{\Phi}^{\pm}}\right)\right) + \mbox{l.o.t}=\frac{{\rm div}\, f_{\mathfrak{h}}^{\pm}}{\partial_1\widehat{\Phi}^{\pm}},
\end{equation}
where ${f}_{\mathfrak{h}}^{\pm}=(f_H^{1\pm}-f_H^{2\pm}\partial_2\widehat{\Psi}^{\pm} ,f_H^{2\pm}\partial_1\widehat{\Phi}^{\pm})$. Both equations
\eqref{S} and equations \eqref{xi} do not need boundary conditions on $x_1=0$ because $\hat{w}_1^{\pm}|_{x_1=0}=0$. Omitting then standard arguments of the energy method, we finally deduce the estimate
\begin{equation}
\label{d1U'}
\|\partial_1\dot{U}_c(t)\|^2_{L^2(\Omega)}\leq C\bigg\{\|f\|^2_{H^1(\Omega_T)} +\nt \dot{U}(t)\nt^2_{{\rm tan},1}+ \int\limits_0^t\nt\dot{U}(s)\nt^2_{H^1(\Omega)}{\rm d}s \bigg\},
\end{equation}
where $\dot{U}_c=(\dot{H}_N^+,\dot{H}_N^-,\dot{S}^+,\dot{S}^-)$. Estimates \eqref{d1Un} and \eqref{d1U'} yield \eqref{d1U}.

As was already noted above, for the contact discontinuity the symbol associated with free boundary is not elliptic. For the linear problem \eqref{28}--\eqref{30} this means that from the boundary conditions \eqref{29} we are not able to express the time-space gradient $(\partial_t\varphi ,\partial_2\varphi )$ of the front through the trace $\dot{U}_{|x_1=0}$. However, as in the nonrelativistic case \cite{MTT1,MTT2}, the last boundary condition is resolvable for the ``material'' derivative $\partial_0\varphi =\partial_t\varphi +\hat{v}_2^+\partial_2\varphi$ (since $\hat{w}_1^+|_{x_1=0}=0$, the derivative $\partial_0$ coincides on the boundary with the material derivative $\partial_t+(\hat{w}^+\cdot\nabla )$ in the reference frame moving with the discontinuity). Then, repeating the arguments from \cite{MTT1,MTT2}, we can obtain the auxiliary estimate
\begin{equation}\label{mat}
\begin{split}
\|\dot{U}(t)\|^2_{L^2(\Omega)}+ & \|\partial_0\dot{U}(t)\|^2_{L^2(\Omega)}+\|\varphi(t)\|^2_{L^2(\Omega)} \\ & \leq \varepsilon C\nt\dot{U}(t)\nt^2_{H^1(\Omega)} +\frac{C}{\varepsilon}\bigg\{ \|f\|^2_{H^1(\Omega_T)}+\int\limits_0^tI(s){\rm d}s \bigg\},
\end{split}
\end{equation}
where $\varepsilon >0$ is an arbitrary constant (here $C$ depends on  $\varepsilon$) and
\[
I(t)=\nt\dot{U}(t)\nt^2_{H^1(\Omega)}+
\|\varphi(t)\|^2_{L^2(\Omega)}+\|\partial_t\varphi(t)\|^2_{L^2(\Omega)}+\|\partial_2\varphi(t)\|^2_{L^2(\Omega)}.
\]

It is clear that having in hand estimates for $\partial_0\dot{U}$ and $\partial_2\dot{U}$ we can derive a corresponding estimate for $\partial_t\dot{U}$. That is,
it remains to get estimates for $\partial_2\dot{U}$ as well as estimates for $\partial_t\varphi$ and $\partial_2\varphi $. Combining these estimates with estimates
 \eqref{d1U} and \eqref{mat} (for a suitable choice of a sufficiently small $\varepsilon$) could give us the energy inequality
\begin{equation}\label{ener}
\begin{split}
I(t) \leq \bigg\{ \|f\|^2_{H^1(\Omega_T)}+\int\limits_0^tI(s){\rm d}s \bigg\}.
\end{split}
\end{equation}
Applying then Gronwall's lemma, this inequality implies the a priori estimate \eqref{main_estL'}. Returning to inhomogeneous boundary conditions, we obtain the a priori estimate \eqref{main_estL}.

The derivation of the auxiliary estimates \eqref{d1U} and \eqref{mat} is purely technical and does not need some new ideas for the problem for contact discontinuity. While deriving  \eqref{d1U} and \eqref{mat} we also do not use the principal Rayleigh--Taylor condition \eqref{RTL}. Obtaining the estimate for $\partial_2\dot{U}$ (in passing we get an estimate for $\partial_2\varphi $), on the contrary, plays the crucial role for deriving the energy inequality  \eqref{ener}.

The derivation of the estimate for $\partial_2\dot{U}$ has no principal differences from that in the nonrelativistic case \cite{MTT1,MTT2}. For understanding the possibility to adapt the arguments from  \cite{MTT1,MTT2} to the relativistic problem, we have only to pay attention to the quadratic form with the boundary matrix $\mathfrak{A}_1$ (see \eqref{A1}) computed on the basic state \eqref{a21}. Let $\widehat{\mathfrak{A}}_1:=\mathfrak{A}_1|_{U^{\pm}=\widehat{U}^{\pm},\;\varphi =\hat{\varphi}}$. Then
\begin{equation}\label{qf}
\big(\widehat{\mathfrak{A}}_1\dot{U}\cdot\dot{U}\big)\big|_{x_1=0} = \big[\big(\{A_1(\widehat{U}+\bar{U}) -A_0(\widehat{U}+\bar{U})\partial_t\varphi  - A_2(\widehat{U}+\bar{U})\partial_2\varphi  \}\dot{U}\cdot\dot{U}\big)\big],
\end{equation}
where the jump $\big[(\cdot ) (\widehat{U},\bar{U},\dot{U})\big] = (\cdot ) (\widehat{U}^+,\bar{U}^+,\dot{U}^+)|_{x_1=0} - (\cdot )(\widehat{U}^-,\bar{U}^-,\dot{U}^-)|_{x_1=0}$.

For calculating the quadratic form  \eqref{qf} we should understand the structure of the matrix
\[
A_1({U}) -A_0({U})\partial_t\varphi  - A_2({U})\partial_2\varphi \quad \mbox{for}\quad \partial_t\varphi =v_N:=v_1-v_2\partial_2\varphi .
\]
To this end one needs to apply the following decomposition for the matrices in system  \eqref{8}:
\[
A_j=v_jA_0+G_j
\]
(recall that we consider the 2D case, i.e.,  $j=1,2$), where
\[
G_j=
\begin{pmatrix}
0 & e_j^{\sf T}-v_jv^{\sf T} &0 &0\\[6pt]
e_j-v_jv & \mathcal{G}_j &{\mathcal{N}_j}^{\sf T} &0\\
0& \mathcal{N}_j & 0 & 0 \\
0& 0 & 0 &0
\end{pmatrix},
\]
\begin{multline*}
\mathcal{G}_j= v_j\left\{ 2\frac{B^2}{\Gamma}\,v\otimes v -\frac{(v\cdot H)}{\Gamma}
\bigl( v\otimes H + H\otimes v\bigr)\right\}\\
+\frac{H_j}{\Gamma}\left\{ \frac{1}{\Gamma^2}
\bigl( v\otimes H + H\otimes v\bigr) -2(v\cdot H)(I-v\otimes v) \right\}
\\
+\frac{(v\cdot H)}{\Gamma}\left( H\otimes e_j + e_j \otimes H\right) -\frac{B^2}{\Gamma}
\left( v\otimes e_j + e_j \otimes v\right).
\end{multline*}
Then, for $\partial_t\varphi =v_N$ we have
\[
A_1({U}) -A_0({U})\partial_t\varphi  - A_2({U})\partial_2\varphi = G_N(U,\varphi ): = G_1({U}) - G_2({U})\partial_2\varphi .
\]
It is easy to see that
\[
G_N=
\begin{pmatrix}
0 & N^{\sf T}-v_Nv^{\sf T} &0 & 0\\[6pt]
N-v_Nv & \mathcal{G}_N &{\mathcal{N}_N}^{\sf T} & 0\\
0& \mathcal{N}_N & 0 & 0\\
0& 0 & 0 &0
\end{pmatrix},
\]
where $N=(1,-\partial_2\varphi )$,
\[
\mathcal{N}_N= \frac{1}{\Gamma}\,b\otimes N -\frac{v_N}{\Gamma}\, b\otimes v - \frac{H_N}{\Gamma^2}\,I,\quad H_N=H_1-H_2\partial_2\varphi
\]
and
\begin{multline*}
\mathcal{G}_N= v_N\left\{ 2\frac{B^2}{\Gamma}\,v\otimes v -\frac{(v\cdot H)}{\Gamma}
\bigl( v\otimes H + H\otimes v\bigr)\right\}\\
+\frac{H_N}{\Gamma}\left\{ \frac{1}{\Gamma^2}
\bigl( v\otimes H + H\otimes v\bigr) -2(v\cdot H)(I-v\otimes v) \right\}
\\
+\frac{(v\cdot H)}{\Gamma}\left( H\otimes N + N \otimes H\right) -\frac{B^2}{\Gamma}
\left( v\otimes N + N \otimes v\right).
\end{multline*}

By virtue of $\partial_t\hat{\varphi}=\hat{v}_{N}^{\pm}|_{x_1=0}$ (see \eqref{a12'}), one gets
\[
\big(\widehat{\mathfrak{A}}_1\dot{U}\cdot\dot{U}\big)\big|_{x_1=0} =\big[\big(G_N(\widehat{U},\hat{\varphi} )\dot{U}\cdot\dot{U}\big)\big].
\]
Since $[\hat{v}]=[\widehat{H}]=0$ and $[\dot{u}]=0$ (see \eqref{a12'}, \eqref{vu'} and \eqref{29} for $g=0$), we have
\[
\begin{split}
\frac{1}{2}\big[\big(G_N(\widehat{U},\varphi )\dot{U}\cdot\dot{U}\big)\big] &=\big[ \dot{p}\dot{u}_N -\hat{v}_N(\hat{v}\cdot \dot{u})\dot{p}+(\widehat{\mathcal{N}}_N\dot{u}\cdot\dot{H})\big],\\
& =\big[\widehat{\Gamma}\dot{p}\dot{v}_N + \dot{v}_N (\hat{b}\cdot\dot{H})-(\widehat{H}_N/\widehat{\Gamma}^2)(\dot{u}\cdot\dot{H})\big],
\end{split}
\]
where $\dot{u}_{N}^{\pm}=\dot{u}_1^{\pm}-\dot{u}_2^{\pm}\partial_2\widehat{\Psi}^{\pm}$ and $\widehat{\mathcal{N}}_N^{\pm}|_{x_1=0}=\mathcal{N}_N(\widehat{U}_{|x_1=0}^{\pm},\hat{\varphi })$. Omitting simple algebra and taking into account the boundary condition $[\dot{v}]=0$ (see \eqref{29} for $g=0$) and assumptions \eqref{a12'}, we finally obtain
\begin{equation}\label{qf1}
\begin{split}
-\frac{1}{2}\big( & \widehat{\mathfrak{A}}_1\dot{U}\cdot  \dot{U}\big)\big|_{x_1=0}  \\ & =\widehat{\Gamma}^+\big\{
\big( \widehat{H}_N^+\dot{v}_2^+-\widehat{H}_2^+\dot{v}_N^+\big)\big( (1-\hat{\sigma}^2)[\dot{H}_{\tau}] + \hat{\sigma}^2[\dot{H}_N]\big)-\dot{v}_N^+[\dot{p}]\big\}\big|_{x_1=0},
\end{split}
\end{equation}
where $\hat{\sigma}=\partial_t\hat{\varphi}/\sqrt{1+(\partial_2\hat{\varphi})^2}$ is the speed of the unperturbed discontinuity  (actually, the explicit form of the coefficients appearing in front of   $[\dot{H}_{\tau}]$ and $[\dot{H}_N]$ is of no interest for  forthcoming arguments).

The principal difference of the quadratic form  \eqref{qf1} from that for the nonrelativistic contact discontinuity  \cite{MTT1,MTT2} is the appearance of the term connected with the presence of $[\dot{H}_N]$. Precisely for this reason we had to modify the procedure of passing to homogeneous boundary conditions and, unlike  \cite{MTT1,MTT2}, we loose one derivative from the source term $f$ in the a priori estimate \eqref{main_estL}. Thanks to this modification we have condition \eqref{HN"} implying
\begin{equation}\label{qf2}
-\frac{1}{2}\big(  \widehat{\mathfrak{A}}_1\dot{U}\cdot  \dot{U}\big)\big|_{x_1=0}   =\widehat{\Gamma}^+\big\{
 (1-\hat{\sigma}^2)\big( \widehat{H}_N^+\dot{v}_2^+-\widehat{H}_2^+\dot{v}_N^+\big)[\dot{H}_{\tau}] -\dot{v}_N^+[\dot{p}]\big\}\big|_{x_1=0}.
\end{equation}

Further arguments totally coincide with those from  \cite{MTT1,MTT2}. For the reader's convenience we shortly comment obtaining the estimate for $\partial_2\dot{U}$. As in \cite{MTT2}, it is convenient to pass from the unknown  $\dot{U}$ to $V= (V^+,V^-)$, with
\[
V^{\pm}=(\dot{p}^{\pm},\dot{v}_N^{\pm},\dot{v}_2^{\pm},\dot{H}_N^{\pm},\dot{H}_{\tau}^{\pm},\dot{S}^{\pm})
\]
(recall that $\dot{U}^{\pm}=(\dot{p}^{\pm},\dot{u}^{\pm},\dot{H}^{\pm},\dot{S}^{\pm})$). The passage to  $V$ is not necessary (it was not done in \cite{MTT1})
but it simplifies the quadratic form with the boundary matrix for the $x_2$--derivatives of the unknown. It is easy to write down the transition matrix $J ={\rm diag}(J^+,J^-)$ depending on the basic state \eqref{a21} for which $\dot{U}^{\pm}=J^{\pm}V^{\pm}$. However, its concrete form is of no interest for deriving estimates. After passing to  $V$ and multiplying system \eqref{28} from the left by the matrix $J^{\sf T}$ we again have a symmetric hyperbolic system. Differentiating this system with respect to  $x_2$ and applying standard arguments of the energy method which are omitted here, we come to the inequality
\begin{equation}
\int\limits_{\Omega}\bigl({B}_0\partial_2V\cdot \partial_2V \bigr){\rm d}x +2 \int\limits_{\partial\Omega_t}Q{\rm d}x_2{\rm d}s \leq C\bigg\{ \|f\|^2_{H^1(\Omega_T)}+\int\limits_0^tI(s){\rm d}s \bigg\},
\label{q2}
\end{equation}
where $B_0=J^{\sf T}\widehat{\mathfrak{A}}_0J,\quad \widehat{\mathfrak{A}}_0={\rm diag}\big(A_0(\widehat{U}^++\bar{U}^+),A_0(\widehat{U}^-+\bar{U}^-)\big)$ and, in view of \eqref{qf2},
\[
Q=-\frac{1}{2}\big(  \widehat{\mathfrak{A}}_1\partial_2V\cdot \partial_2V \big)\big|_{x_1=0}   =\widehat{\Gamma}^+\big\{
 (1-\hat{\sigma}^2)R[\partial_2\dot{H}_{\tau}] -\partial_2\dot{v}_N^+[\partial_2\dot{p}]\big\}\big|_{x_1=0},
\]
with $R= \big(\widehat{H}_N^+\partial_2\dot{v}_2^+-\partial_2\widehat{H}_2^+\dot{v}_N^+\big)\big|_{x_1=0}$.

Following then  \cite{MTT1,MTT2},  from \eqref{dHN} at $x_1=0$ we deduce
\begin{equation}
R= -\partial_0\dot{H}_N^+  + ({f}^+_{H}\cdot \widehat{N}^+)+ \mbox{l.o.t.}\quad \mbox{on}\ \partial\Omega_T.
\label{R}
\end{equation}
Using the first, fourth and fifth boundary conditions in \eqref{29} (for $g=0$) and substituting the expression for $R$ from \eqref{R} into the quadratic form $Q$, we get
\begin{equation}
Q=\frac{1}{2}\partial_t\big\{ \widehat{\Gamma}^+[\partial_1\hat{p}](\partial_2\varphi)^2\big\}\big|_{x_1=0} +{\rm coeff}\,\partial_0\dot{H}_N^+\,\partial_2\varphi |_{x_1=0} +P,
\label{Q}
\end{equation}
where ${\rm coeff}=(1-\hat{\sigma}^2) [\partial_1\widehat{H}_{\tau}]$ (in fact, the explicit form of this coefficient is of no interest) and $P$ is the quadratic form on the boundary which is the sum of lower-order terms (in some sense) whose boundary integrals can be estimated by standard methods (for instance, by passing from the integral over the boundary $\partial\Omega_t$ to a volume integral over  $\Omega_t$ and integrating by parts, see \cite{MTT1,MTT2}). The first term in the right-hand side of  \eqref{Q} plays a crucial role because, in view of the inequality  $\widehat{\Gamma}^+>1$ and the Rayleigh--Taylor sign condition \eqref{RTL},
after the time integration it gives a control on the norm $\|\partial_2\varphi (t)\|_{L^2(\partial\Omega )}$ (it is only important here that the positive $\widehat{\Gamma}^+$ is separated from zero). Estimating the boundary integral of the second term in the right-hand side of equality \eqref{Q} is most complicated in the deduction of the a priori estimate  \eqref{main_estL'} (in the sense of the usage of nontrivial ideas). To this end one uses the control on the norm $\|\partial_2\varphi (t)\|_{L^2(\partial\Omega )}$, passage to a volume integral, integration by parts, the Young inequality, etc. Since for our problem these arguments totally coincide with those for the nonrelativistic case, we just again refer the reader to \cite{MTT1,MTT2}.

As a result, as in  \cite{MTT1}, we obtain an $L^2$ estimate for $\partial_2\dot{U}$ and $\partial_2\varphi$ whose combination with estimates  \eqref{d1U} and \eqref{mat} as well as with the $L^2$ estimate for $\partial_t\varphi$ following from the last boundary condition in \eqref{29} (for $g=0$) yields the energy inequality \eqref{ener}. This inequality implies the a priori estimate \eqref{main_estL'}, and after returning to inhomogeneous boundary conditions we obtain the a priori estimate \eqref{main_estL}. As in \cite{MTT1}, the proof of the existence of solutions to the linearized problem \eqref{28}--\eqref{30} is performed
by noncharacteristic ``strictly dissipative'' regularization of this problem (see \cite{MTT1} for more details). It is clear that the uniqueness of a solution follows from the basic a priori estimate \eqref{main_estL}.

\section{Sketch of the proof of Theorem \ref{t1}}

The proof of Theorem \ref{t1} is totally analogous to that for the nonrelativistic case \cite{MTT2}. Therefore, referring the reader to  \cite{MTT2}, here we just briefly outline the proof of the existence theorem for the nonlinear problem \eqref{11.1}--\eqref{13.1}. The uniqueness of a solution to the nonlinear problem is proved by a standard way (see, e.g., \cite{ST}) and follows from the basic a priori estimate \eqref{main_estL} for the linearized problem.

Before proving existence of solutions we should write down compatibility conditions for the initial data. After that, assuming that the initial data are compatible up to  order $\mu$ (see \cite{MTT2}), we define a so-called approximate solution  $((U^{a+},U^{a-}),\varphi^a)\in H^{\mu +1}(\Omega_T)\times H^{\mu +1}(\partial\Omega_T)$ satisfying the boundary conditions, the initial data and the equations
\begin{equation}
\partial_t^j\mathbb{L}(U^{a\pm},\Psi^{a\pm} )|_{t=0}=0 \quad\mbox{in}\ \Omega\ \mbox{for}\ j=0,\ldots , \mu -1,
\label{62}
\end{equation}
where  ${\Psi}^{a\pm} =\chi (\pm x_1)\varphi^a$. We then introduce the vector-functions
\[
f^{a\pm}:=\left\{ \begin{array}{lr}
- \mathbb{L}(U^{a\pm},{\Psi}^{a\pm} ) & \quad \mbox{for}\ t>0,\\
0 & \ \mbox{for}\ t<0,\end{array}\right.
\]
for which, by virtue of \eqref{62}, one can show that
\[
\|f^{a\pm}\|_{H^{\mu}(\Omega_T)}\leq \delta_0 (T),
\]
where the constant $\delta_0(T)\rightarrow 0$ as $T\rightarrow 0$.

After that in problem \eqref{11.1}--\eqref{13.1} we shift the unknown to the approximate solution. Denoting the shifted unknown as the original unknown $((U^{+},U^{-}),\varphi)$, we obtain for it the following problem with zero initial data:
\begin{align}
 \mathcal{ L}(U^{\pm} ,{\Psi}^{\pm})=f^{a\pm} &\quad\mbox{in}\ \Omega_T, \label{69'}\\[3pt]
 \mathcal{ B}(U^+,U^- ,\varphi )=0 &\quad\mbox{on}\ \partial\Omega_T,\label{70'}
\\[3pt]
 ((U^+,U^-),\varphi )=0 &\quad  \mbox{for}\ t<0,\label{71'}
\end{align}
where
\begin{align*}
 &\mathcal{ L}(U^{\pm} ,{\Psi}^{\pm} ):=\mathbb{L}(U^{a\pm} +U^{\pm}  ,{\Psi}^{a\pm}+{\Psi}^{\pm} ) -
\mathbb{L}(U^{a\pm} ,{\Psi}^{a\pm}),\\
 &\mathcal{ B}(U^+,U^- ,\varphi):=\mathbb{B}(U^{a+} +U^+ ,U^{a-} +U^-,\varphi^a+\varphi ).
\end{align*}

The existence of solutions of problem  \eqref{69'}--\eqref{71'} is proved by suitable Nash--Moser iterations. The general description of the Nash--Moser method for problems like \eqref{69'}--\eqref{71'} can be found in \cite{Sec} (see also references therein). The main idea is to solve the nonlinear equation
$F(u)=0$ by the iteration scheme
\[
F'(S_{\theta_n}u_n)(u_{n+1}-u_{n} )=-F(u_n ),
\]
where $F'$ is the linearization of $F$ and $S_{\theta_n}$ is a sequence of smoothing operators, with  $S_{\theta_n}\rightarrow I$ as $n\rightarrow \infty$. This scheme is the classical Newton scheme if $S_{\theta_n}= I$. The application of smoothing operators is needed for compensating the loss of derivatives in a priori estimates for the linearized problem. For problem \eqref{69'}--\eqref{71'} (as well as for corresponding problems in \cite{MTT2,ST,T09,Tcpam}) the Nash--Moser scheme is not quite standard because at each its step we have to construct an intermediate state $u_{n+1/2}$ satisfying the same restrictions as the basic \eqref{a21} about we have performed linearization, i.e., conditions \eqref{a5}--\eqref{jc1'}.

The basic a priori estimate \eqref{main_estL} is not enough for the proof of convergence of the Nash--Moser iterations. First, it is clear that we need a corresponding estimate in high norms, i.e., an estimate in the Sobolev spaces $H^s$, with $s$ being large enough, for estimating supremum norms through $H^s$ norms by embedding theorems. Second, for the proof of convergence of the Nash--Moser iterations we need a more delicate estimate showing not only how many derivatives are lost from the source terms to the solution but also the number of derivatives lost from the coefficients of the problem (i.e., from the basic state \eqref{a21}) to the solution.
Following \cite{MTT2} and taking into account the remarks made in the previous section for our relativistic case, we write down here such a delicate estimate. Such kind of estimates are usually called tame estimates.

\begin{theorem}
Let $T>0$ and $s\in \mathbb{N}$, with $s\geq 3$.  Assume that the basic state $(\widehat{U} ,\hat{\varphi})\in H^{s+3}(\Omega_T )\times H^{s+3}(\partial\Omega_T)$
satisfies assumptions \eqref{a22}--\eqref{jc1'},  the Rayleigh--Taylor sign condition \eqref{RTL} and the inequality
\begin{equation}
\|\widehat{U}\|_{H^6(\Omega_T )} +\|\hat{\varphi} \|_{H^{6}(\partial\Omega_T)}\leq \widehat{K},
\label{37}
\end{equation}
where $\widehat{K}>0$ is a constant. Let also the data $(f ,g)\in
H^{s+1}(\Omega_T )\times H^{s+1}(\partial\Omega_T)$ vanish in the past. Then there exists a positive constant $K_0$ that does not depend on $s$ and $T$
and there exists a constant $C(K_0) >0$ such that, if $\widehat{K}\leq K_0$, then there exists a unique solution $(\dot{U} ,\varphi)\in H^{s}(\Omega_T )\times H^{s}(\partial\Omega_T)$ to problem \eqref{28}--\eqref{30} that obeys the tame a priori  estimate
\begin{equation}
\begin{split}
\|\dot{U}\|_{H^s(\Omega_T )}+ & \|\varphi\|_{H^{s}(\partial\Omega_T)}\leq  C(K_0)\Bigl\{
\|f\|_{H^{s+1}(\Omega_T )}+ \|g \|_{H^{s+1}(\partial\Omega_T)} \\
 &+\bigl( \|f\|_{H^{4}(\Omega_T )}+ \| g\|_{H^{4}(\partial\Omega_T)} \bigr)\bigl(
\|\widehat{U}\|_{H^{s+3}(\Omega_T )}+\|\hat{\varphi}\|_{H^{s+3}(\partial\Omega_T)}\bigr)\Bigr\}
\end{split}
\label{38}
\end{equation}
for a sufficiently short time $T$ (the constant $C(K_0)$ depends also on the fixed constants $\bar{p}$, $\nu$, ${\kappa}$ and $\epsilon$ from \eqref{a5}, \eqref{ls}, \eqref{cdass} and \eqref{RT1}).
\label{t3.1}
\end{theorem}

The tame estimate \eqref{38} differs from that in the nonrelativistic case \cite{MTT2} only by the fact that in it, as in the basic estimate \eqref{main_estL}, we loose one derivative from the source term $f$. As was noted above, this does not influence on the final result for the nonlinear problem because exactly the estimate like \eqref{38} was used \cite{MTT2} for the proof of convergence of the Nash--Moser iterations. The deduction of estimate \eqref{38} itself is performed following the same procedure as that for deriving the basic a priori estimate \eqref{main_estL}. The only difference is that we have to apply Moser-type calculus inequalities (following form the Gagliardo--Nirenberg inequality) for estimating commutators, i.e., lower-order terms appearing after the higher-order differentiation of system \eqref{28} and the boundary conditions  \eqref{29} (see \cite{MTT2}).

This completes our short comments about the proof of Theorem \ref{t1} and for more details we again refer the reader to the proof of the analogous theorem in \cite{MTT2}.

\paragraph{Acknowledgements}
This work was supported by Mathematical Center in Akademgorodok and RFBR (Russian Foundation
for Basic Research) grant No. 19-01-00261-a.

\end{document}